\documentclass[11pt,reqno]{amsart}

\usepackage[top=1.5in, bottom=2in, left=1.2in, right=1.2in]{geometry}

\usepackage{hyperref}
\hypersetup{%
  pdftitle={\@title},
  pdfauthor={\@author},
  unicode=true,
  colorlinks=true,
  linkcolor=blue,
  citecolor=blue,
  filecolor=blue,
  urlcolor=blue
}

\newtheorem{theorem}{Theorem}
\newtheorem{proposition}{Proposition}
\newtheorem{lemma}{Lemma}

\usepackage[utf8]{inputenc}
\usepackage[T1]{fontenc}
\usepackage[english]{babel}

\usepackage{graphicx}

\usepackage[usenames,dvipsnames]{xcolor}

\usepackage[round]{natbib}

\usepackage{amsmath,amsfonts,amssymb,amsthm,amscd}

\usepackage{nameref}
\usepackage[nameinlink]{cleveref}
\crefname{equation}{equation}{equations}
\crefname{assumption}{assumption}{assumptions}
\creflabelformat{enumi}{(#2#1#3)}

\usepackage{mathtools}

\providecommand\given{}
\newcommand\SetSymbol[1][]{
  \nonscript\,#1:\nonscript\,\mathopen{}\allowbreak}
\DeclarePairedDelimiterX\Set[1]{\lbrace}{\rbrace}%
{ \renewcommand\given{\SetSymbol[]} #1 }

\allowdisplaybreaks

\usepackage{stmaryrd}
\usepackage{mathrsfs}
\usepackage[inline]{enumitem}
\usepackage{bbm}
\usepackage{multirow}
\usepackage{booktabs}

\DeclareMathOperator{\integers}{\mathbb{N}}
\DeclareMathOperator{\reals}{\mathbb{R}}

\DeclareMathOperator{\complex}{\mathbb{C}}

\DeclareMathOperator{\supp}{supp}

\DeclareMathOperator{\Esp}{\mathbb{E}}

\DeclareMathOperator{\Ga}{Ga}

\DeclareMathOperator{\e}{e}

\DeclareMathOperator{\Ind}{\mathbbm{1}}

\DeclareMathOperator{\iid}{\overset{\mathrm{i.i.d}}{\sim}}

\newcommand{\sieve}{\Theta_n}
\newcommand{\sievec}{\Theta_n^c}

\newcommand{\spx}{\mathcal{X}}

\newcommand{\piq}{\Pi_{*}}

\title{Some aspects of symmetric Gamma process mixtures}
\author{Zacharie Naulet\and {\'E}ric Barat}
\date{\today}

\begin{document}

\begin{abstract}
  In this article, we present some specific aspects of symmetric Gamma process
mixtures for use in regression models. We propose a new Gibbs sampler for
simulating the posterior and we establish adaptive posterior rates of
convergence related to the Gaussian mean regression problem.

\end{abstract}

\maketitle

\section{Introduction}
\label{sec:introduction}

Recently, interest in a Bayesian nonparametric approach to the sparse regression
problem based on mixtures emerged from works of \citet{Abramovich2000},
\citet{DeJongeVanZanten2010} and \citet{Wolpert2011}. The idea is to model the
regression function as
\begin{equation}
  \label{eq:32}
  f(\cdot) = \int_{\spx} K(x;\cdot) Q(dx),\quad Q \sim \Pi_{*},
\end{equation}
where $K : \spx \times \reals^d \rightarrow \reals$ is a jointly measurable
kernel function, and $\Pi_{*}$ a prior distribution on the space of signed
measure over the measurable space $\spx$. Although the model \eqref{eq:32} is
popular in density estimation \citet{EscobarWest1995, MuellerErkanliWest1996,
  Ghosal2007, Shen2013, Canale2013} and for modeling hazard rates in Bayesian
nonparametric survival analysis \citet{LoWeng1989, PeccatiPruenster2008,
  DeBlasiPeccatiPruensterEtAl2009, IshwaranJames2012, LijoiNipoti2014}, it seems
that much less interest has been shown in regression.

Perhaps the little interest for mixture models in regression is due to the lack
of variety in the choice of algorithms available, and in the insufficiency of
theoretical posterior contraction results. To our knowledge, the sole algorithm
existing for posterior simulations is to be found in \citet{Wolpert2011}, when
the mixing measure $Q$ is a \textit{Lévy process}. On the other hand, The only
contraction result available is to be found in \citet{DeJongeVanZanten2010} for
a suitable semiparametric mixing measure.

Indeed, both designing an algorithm or establishing posterior contraction
results heavily depends on the choice of $K$ and $\Pi_{*}$ in \cref{eq:32}; but
above all also on the observation model we consider. This last point makes the
study of mixtures in regression nasty to handle because of the diversity of
observation models possible. In this article, we focus on the situation when $Q$
is a \textit{symmetric Gamma process} to propose both a new algorithm for
posterior simulations and posterior contraction rates results.

In the first part of the paper, we propose a Gibbs sampler to get samples from
the posterior distribution of symmetric Gamma process mixtures. The algorithm is
sufficiently general to be used in all observation models for which the
likelihood function is available. We begin with some preliminary theoretical
result about approximating symmetric Gamma process mixtures, before stating the
general algorithm. Finally, we make an empirical study of the algorithm, with
comparison with the RJMCMC algorithm of \citet{Wolpert2011}.

The second part of the paper is devoted to posterior contraction rates
results. We consider the mean regression model with normal errors of unknown
variance, and two types of mixture priors: location-scale and
location-modulation. The latter has never been studied previously, mainly
because it is irrelevant in density estimation models. However, we show here
that it allows to get better rates of convergence than location-scale mixtures,
and thus might be interesting to consider in regression.

\section{Symmetric Gamma process mixtures}
\label{sec:symm-gamma-rand-1}

Let $(\Omega, \mathcal{E}, \mathbb P)$ be a probability space and
$(\spx,\mathcal{A})$ be a measurable space. We call a mapping
$Q : \Omega\times \mathcal{A}\rightarrow \reals\cup \Set{\pm\infty}$ a
\textit{signed random measure} if $\omega\mapsto Q(\omega,A)$ is a random
variable for each $A\in \mathcal{A}$ and if $A\mapsto Q(\omega,A)$ is a signed
measure for each $\omega\in \Omega$.

Symmetric Gamma random measures are infinitely divisible and independently
scattered random measures (the terminology Lévy base is also used in
\citet{Barndorff-Nielsen2004}, and Lévy random measure in \citet{Wolpert2011}),
that is, random measures with the property that for each disjoint
$A_1,\dots,A_k \in \mathcal{A}$, the random variables $Q(A_1),\dots,Q(A_k)$ are
independent with infinitely divisible distribution. More precisely, given
$\alpha,\eta > 0$ and $F$ a probability measure on $\spx$, a symmetric Gamma
random measure assigns to all measurable set $A \in \mathcal{A}$ random
variables with distribution $\mathrm{SGa}(\alpha F(A), \eta)$ (see
\cref{sec:symm-gamma-distr}). Existence and uniqueness of symmetric Gamma random
measures is stated in \citet{rajput1989spectral}.

In the sequel, we shall always denote by $\piq$ the distribution of a
symmetric Gamma random measure with parameters $\alpha,\eta$ and $F$, and we
refer $\alpha F$ as the base distribution of $Q \sim \piq$, and $\eta$ as the
scale parameter.

\subsection{Location-scale mixtures}
\label{sec:locat-scale-mixt}

Given a measurable \textit{mother} function $g : \reals^d \rightarrow \reals$,
we define the \textit{location-scale} kernel $K_{A}(x) := g(A^{-1}x)$, for all
$x\in \reals^d$ and all $A\in \mathcal{E}$, where $\mathcal{E}$ denote the set
of all $d\times d$ positive definite real matrices. Then we consider symmetric
Gamma location-scale mixtures of the type
\begin{align}
  \label{eq:24}
  f(x;\omega) := \int_{\mathcal{E}\times \reals^d}K_A(x-\mu)\,Q(dA
  d\mu;\omega),\qquad \forall x\in \reals^d,
\end{align}
where $Q : \mathcal{E}\times\reals^d\times \Omega \rightarrow [-\infty,\infty]$
is a symmetric Gamma random measure with base measure $\alpha F$ on
$\mathcal{E}\times \reals^d$, and scale parameter $\eta > 0$. The precise
meaning of the integral in \cref{eq:24} is made clear in
\citet{rajput1989spectral}.

\subsection{Location-modulation mixtures}
\label{sec:locat-modul-mixt}

As in the previous section, given a measurable mother function
$g : \reals^d \rightarrow \reals$, we define the \textit{location-modulation}
kernel $K_{\xi,\phi}(x) := g(x) \cos(\sum_{i=1}^d\xi x_i + \phi)$, for all
$x\in \reals^{d}$, all $\xi \in \reals^d$ and all $\phi \in [0,\pi/2]$. Then we
consider symmetric Gamma location-modulation mixtures of the type
\begin{align}
  \label{eq:26}
  f(x;\omega) := \int_{\reals^d\times
  \reals^d\times[0,\pi/2]}K_{\xi,\phi}(x-\mu)\,Q(d\xi d\mu d\phi;\omega),\qquad
  \forall x\in \reals^d,
\end{align}
where
$Q : \reals^d \times \reals^d \times[0,\pi/2]\times \Omega \rightarrow
[-\infty,\infty]$ is a symmetric Gamma random measure with base measure
$\alpha F$ on $\reals^d \times \reals^d \times [0,\pi/2]$, and scale parameter
$\eta > 0$.

\subsection{Convergence of mixtures}
\label{sec:convergence-mixtures-1}

Given a kernel $K : \spx \times \reals^d \rightarrow \reals$ and a symmetric
Gamma random measure $Q$, it is not clear a priori whether or not the mixture
$y\mapsto \int K(x;y)\, Q(dx)$ converges or not, and in what sense. According to
\citet{rajput1989spectral} (see also \citet{Wolpert2011}),
$y\mapsto \int K(x;y)\, Q(dx)$ converges almost-surely at all $y$ for which
\begin{align*}
  \int_{\reals\times \spx} (1\wedge |u K(x;y)|) |u|^{-1} e^{-|u|\eta}F(dx)
  <+\infty.
\end{align*}
Moreover, from the same references (or also in \citet{KINGMAN1992}), if $M$ is a
complete normed space equipped with norm $\|\cdot\|$, then
$y\mapsto \int K(x;y)\, Q(dx)$ converges almost-surely in $M$ if
\begin{align*}
  \int_{\reals\times \spx} (1\wedge |u|\|K(x;\cdot)\|) |u|^{-1}
  e^{-|u|\eta}F(dx) <+\infty.
\end{align*}
Since by definition $F$ is a probability measure, we have for instance that the
mixtures of \cref{eq:24,eq:26} converges almost surely in $L^{\infty}$ as soon
as $\|K_{A}\|_{\infty} < +\infty$ for $F$-almost every $A\in \mathcal{E}$, or
$\|K_{\xi,\phi}\|_{\infty} < +\infty$ for $F$-almost every
$(\xi,\phi)\in \reals^d\times[0,\pi/2]$.

\section{Simulating the posterior}
\label{sec:algorithm}

In this section we propose a Gibbs sampler for exploration of the posterior
distribution of a mixture of kernels by a symmetric Gamma random measure. The
sampler is based on the series representation of the next theorem, inspired from
a result about Dirichlet processes from \citet{Favaro2012}, adapted to symmetric
Gamma processes. In \cref{thm:2}, we consider $\mathcal{M}(\spx)$ the space of
signed Radon measures on the measurable space $(\spx,\mathcal{A})$. By the
Riesz-Markov representation theorem \citep[Chapter~6]{rudin1974real},
$\mathcal{M}(\spx)$ can be identified as the dual space of $\mathtt C_c(\spx)$,
the space continuous functions with compact support. That said, we endow
$\mathcal{M}(\spx)$ with the topology $\mathcal{T}_v$ of weak-* convergence
(sometimes referred as the topology of \textit{vague} convergence), that is, a
sequence $\Set{\mu_n\in \mathcal{M}(\spx) \given n\in \integers}$ converges to
$\mu \in \mathcal{M}(\spx)$ with respect to the topology $\mathcal{T}_{v}$, if
for all $f\in \mathtt C_c(\spx)$,
\begin{align*}
    \int_{\spx} f(x)\,d\mu_n(x) \rightarrow \int_{\spx} f(x)\,d\mu(x).
\end{align*}
Dealing with prior distributions on $\mathcal{M}(\spx)$, we shall equip
$\mathcal{M}(\spx)$ with a $\sigma$-algebra. Here it is always considered the
Borel $\sigma$-algebra of $\mathcal{M}(\spx)$ generated by $\mathcal{T}_v$.

Before stating the main theorem of this section, we recall that a sequence of
random variables $\Set{X_i \in \spx \given 1\leq i \leq n}$ is a
\textit{P{\'o}lya urn sequence} with base distribution $\alpha F(\cdot)$, where
$F$ is a probability distribution on $(\spx,\mathcal{A})$ and $\alpha > 0$, if
for all measurable set $A\in \mathcal{A}$,
\begin{align*}
  P(X_1 \in A) = F(A),
  \quad P(X_{k+1} \in A\, |\, X_1,\dots,X_k) = F_k(A)/F_k(\spx),\ k=2,\dots,n-1,
\end{align*}
where $F_k := \alpha F + \sum_{i=1}^{k}\delta_{X_i}$. We are now in position to
state the main theorem of this section, which proof is given in
\cref{sec:symm-gamma-distr}.

\begin{theorem}
  \label{thm:2}
  Let $\spx$ be a Polish space with Borel $\sigma$-algebra, $p>0$ be integer,
  $T\sim \mathrm{Ga}(\alpha,\eta)$, independently,
  $J_1,\dots,J_p \iid \mathrm{SGa}(1,1)$, and
  $\Set{X_i \in \spx \given 1\leq i \leq p}$ a P{\'o}lya urn sequence with base
  distribution $\alpha F(\cdot)$, independent of $T$ and of the $J_i$'s. Define
  the random measure, $Q_p := \sqrt{T/p}\sum_{i=1}^p J_i\,\delta_{X_i}$. Then
  $Q_p \overset{d}{\rightarrow} Q$, where $Q$ is a symmetric Gamma random
  measure with base distribution $\alpha F(\cdot)$ and scale parameter
  $\sqrt{\eta}$.
\end{theorem}

\subsection{Convergence of sequences of  mixtures}
\label{sec:convergence-mixtures}

In \cref{thm:2}, we proved weak convergence of the sequence of approximating
measures $(Q_p)_{p\geq 1}$ to the symmetric Gamma random measure, but it is not
clear that mixtures of kernels by $Q_p$ also converge. The next proposition
establish convergence in $L^q$ for general kernels, with $1\leq q < +\infty$,
the proof is similar to the proof of \citet[Theorem~2]{Favaro2012}, thus we
defer it into \cref{sec:proofs}. For any kernel
$K : \spx \times\reals^d \rightarrow \complex$, and any (signed) measure $Q$ on
$(\spx, \mathcal{A})$, we write
\begin{align*}
  f^{(Q)}(y)
  := \int_{\spx} K(x;y)\,Q(dx).
\end{align*}

\begin{proposition}
  \label{pro:3}
  If $x\mapsto K(x;y)$ is continuous for all $x\in \spx$, vanishes outside a
  compact set, and bounded by a Lebesgue integrable function $h$, then for any
  $1 \leq q < +\infty$ we have
  $\lim_{p\rightarrow \infty}\|f^{(Q_p)} - f^{(Q)}\|_q = 0$ almost-surely.
\end{proposition}

Under supplementary assumptions on $K$, we can say a little-more about uniform
convergence of the approximating sequence of mixtures. Assuming that
$y\mapsto K(x;y)$ is in $L^1$ for all $x\in \spx$, we denote by
$(x,u)\mapsto \widehat{K}(x;u)$ the $L^1$ Fourier transform on the second
argument of $(x,y) \mapsto K(x;y)$.

\begin{proposition}
  \label{pro:4}
  Let $y\mapsto K(x;y)$ be in $L^1$ for all $x\in \spx$ and $\widehat{K}$
  satisfies the assumption of \cref{pro:3}. Then
  $\lim_{p\rightarrow \infty}\|f^{(Q_p)} - f^{(Q)}\|_{\infty} = 0$
  almost-surely.
\end{proposition}
\begin{proof}

  We can assume without loss of generality that $f^{(Q_p)}$ and $f^{(Q)}$ are
  defined on the same probability space $(\Omega,\mathcal{F},\mathbb P)$. By
  duality, it is clear that
  $\|f^{(Q_p)}(\cdot;\omega) - f^{(Q)}(\cdot;\omega)\|_{\infty} \leq
  \int_{\reals^d}|\widehat{f}^{(Q_p)}(u,\omega) - \widehat{f}^{(Q)}(u;\omega)|\,
  du$, where $\widehat{f}$ denote the $L^1$ Fourier transform of $f$. Notice
  that by assumptions on $K$, $\widehat{f}^{(Q_p)}$ and $\widehat{f}^{(Q)}$ are
  well-defined for almost all $\omega \in \Omega$ (see
  \cref{sec:convergence-mixtures-1}). Then by Fubini's theorem
  \begin{align*}
    \widehat{f}^{(Q_{p)}}(u;\omega)
    &= \int_{\reals^d}\int_{\spx}K(x;y)\, Q(dx;\omega)\, \e^{-iuy}\,dy\\
    &= \int_{\spx}\int_{\reals^d} K(x;y)\e^{-iuy}\,dy\,Q(dx;\omega)
      = \int_{\spx} \widehat{K}(x;u)\, Q(dx;\omega),
  \end{align*}
  and the conclusion follows from \cref{pro:3}.
\end{proof}

\subsection{General algorithm}
\label{sec:sampling-algorithm}

From \cref{thm:2}, replacing $Q$ by $Q_p$ for sufficiently large $p$, we propose
a P\'olya urn Gibbs sampler adapted from algorithm~8 in \cite{Neal2000}. In the
sequel, we refer to $Q_p$ as the particle approximation of $Q$ with $p$
particles.

Let $Y = (Y_i)_{i=1}^n$ be observations coming from a statistical model with
likelihood function $\mathcal{L}(f | Y)$, where
$f : \reals^d \rightarrow \reals$ is the regression function on which we put a
symmetric Gamma mixture prior distribution. Let $X = (X_i)_{i=1}^p$ be a P\'olya
urn sequence, $J := (J_1,\dots,J_p)$ a sequence of i.i.d. $\mathrm{SGa}(1,1)$
random variables, and $T \sim \Ga(\alpha,\eta)$ independent of $(X_i)_{i=1}^p$
and $J$. We introduce the clustering variables $C := (C_1,\ldots,C_p)$ such that
$C_i=k$ if and only if $X_i=X_k^\star$ where $X^{\star}:=X_1^{\star},\ldots$
stands for unique values of $(X_i)_{i=1}^p$. In the sequel, $C_{-i}$ stands for
the vector obtained from removing the coordinate $i$ to $C$, and the same
definition holds for $J$ \textit{mutatis mutandis}. Given $J,C,X,T$ and a
measurable kernel $K : \spx \times\reals \rightarrow \reals$ we construct $f$ as
\begin{align*}
  f(x) = \sqrt{\frac{T}{p}}\sum_{i=1}^pJ_i\, K(X_i;x).
\end{align*}
We propose the following algorithm. At each iteration, successively sample from
:
\begin{enumerate}
  \item $C_i \vert C_{-i},Y,X,J,T$, for $1\le i \le p$. Let
  $n_{k,i}=\#_{\stackrel{1\le l \le n}{l\ne i}}\lbrace C_l=k\rbrace$,
  $\kappa^{(p)}$ the number of distinct $X_k$ values and $\kappa_0$ a chosen
  natural,
  \begin{align*}
    C_i & \stackrel{\text{ind}}{\sim} \sum_{k=1}^{\kappa^{(p)}}
    n_{k,i}\,\mathcal{L}_{k,i}(X, J,
    T|Y)\,\delta_{k}(\cdot)
    + \frac{\alpha}{\kappa_0} \sum_{k=1}^{\kappa_0}
    \mathcal{L}_{k+\kappa^{(p)},i}(X,J,T|Y)\,\delta_{k+\kappa^{(p)}}(\cdot),
  \end{align*}
  where $\mathcal{L}_{k,i}(X,J,T|Y)$ stands for the likelihood under hypothesis
  that particle $i$ is allocated to component $k$ (note that the likelihood
  evaluation requires the knowledge of whole distribution $F$ under any
  allocation hypothesis).
  \item $X\vert C,Y,J,T$. Random Walk Metropolis Hastings on parameters.
  \item $J_i \vert J_{-i},K,Y,X,T$, for $1\le i \le p$. Independent Metropolis
  Hastings with prior $\mathrm{SGa}(1,1)$ taken as i.i.d. candidate distribution
  for $J_i$. Note that for $n\rightarrow\infty$, the posterior distribution of
  $J_i \vert J_{-i},C,Y,Z$ should be $\mathrm{SGa}(1,1)$, then the number of
  particles $p$ may be monitored using the acceptance ratio of the $J_i$'s.
  \item $T \vert C,Y,X,J$. Random Walk Metropolis Hastings on scale parameter.
\end{enumerate}

\subsection{Assessing the convergence of the Markov Chain}
\label{sec:assess-conv-mark}

The previous algorithm produces a Markov Chain whose invariant distribution is
(an approximation of) the posterior distribution of a symmetric Gamma process
mixture. However, if the Markov Chain is initialized in a region of low
posterior probability mass, we may over-sample this region. To avoid such
over-sampling, we discard the first $n_0$ samples of the chain using Geweke's
convergence diagnostic \citep{Geweke1992}.

More precisely, we monitor the convergence of the chain using the log-likelihood
function. We start the algorithm with Markov Chain initialized at random from
prior distribution. Then after $n\gg n_0$ iterations we compute Geweke's
$Z$-statistic for the log-likelihood using the whole chain; if the statistic is
outside the $95\%$ confidence interval we continue to apply the diagnostic after
discarding $10\%$, $20\%$, $30\%$ and $40\%$ of the chain. If the $Z$-statistic
is still outside $95\%$ confidence interval, the chain is reported as failed to
converge, and we restart the algorithm from a different initialization point.

Once we have discarded the first $n_0$ samples using Geweke's test, we run the
chain sufficiently longer to get an \textit{Effective Sample Size} (ESS) of at
least $1000$ samples, where we measure the ESS through the value of the
log-likelihood at each iteration of the Markov Chain. A thinning of the chain is
not required in general, however, we found in practice that a slight thinning
improves the efficiency of the sampling.

In \cref{fig:assess-mcmc}, we draw some examples of temporal evolution of the
log-likelihood on a simple univariate Gaussian mean regression problem. Here and
after, we always choose step sizes in RWMH steps to achieve approximately 30\%
acceptance rates for each class of updates. Each subfigure represent 10
simulations with random starting point of the Markov Chain, distributed
according to the prior distribution. We draw each subfigure varying the
parameters liable to influence the mixing time of the chain, notably $m$ and the
number of particles. We observe that the speed at which the chain reach
equilibrium is fast, especially when the number of particles is high. This last
remark have to be balanced with the complexity in time of the algorithm which is
$O(mnp)$ for a naive implementation, and, depending on the nature of the
likelihood, can be reduced to $O(mp)$ or $O(mp^2)$.

\begin{figure}[!htb]
  \centering
  \includegraphics[width=0.95\linewidth]{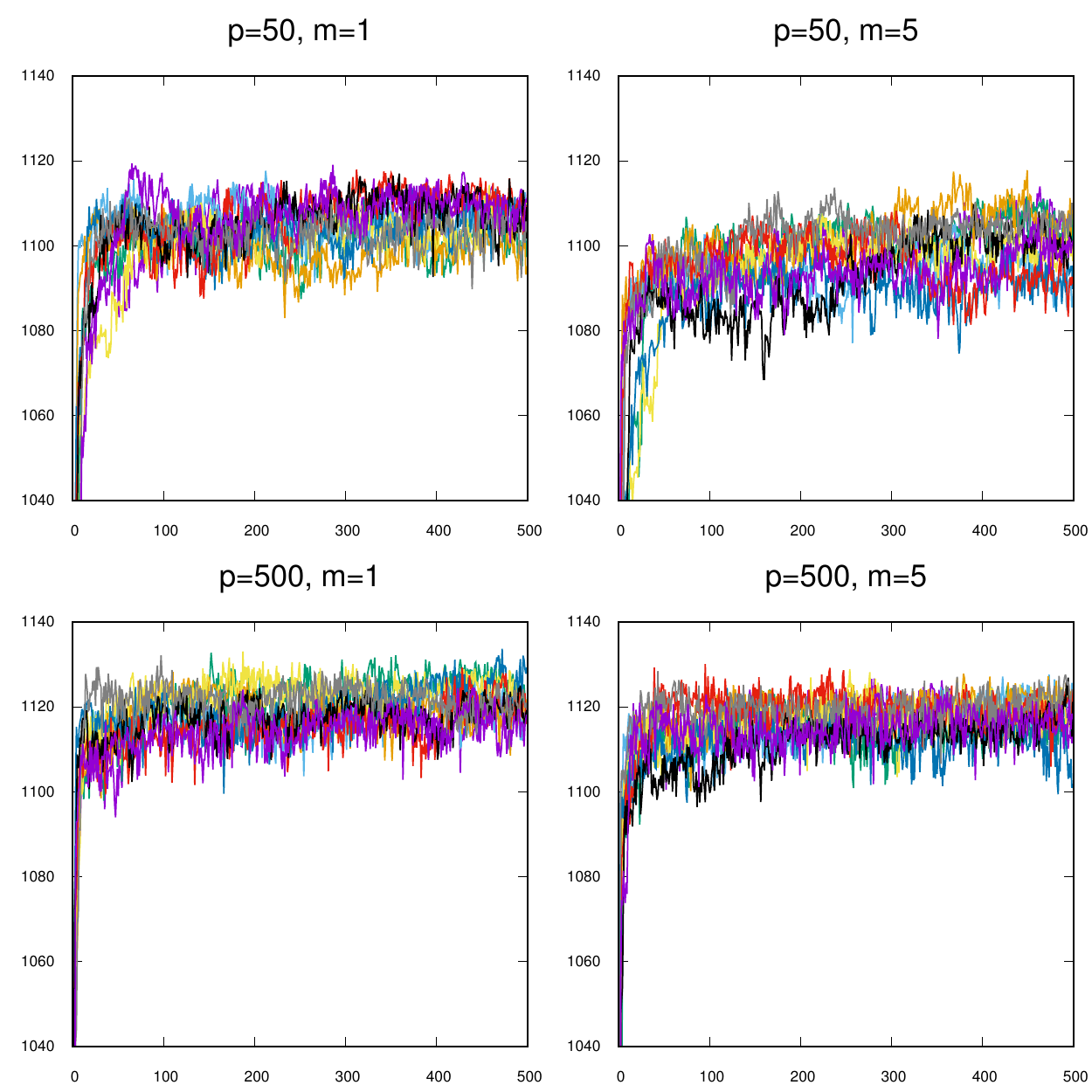}
  \caption{Time evolution of the log-likelihood for different starting point of
    the Markov Chain, chosen according to the prior distribution, and various
    parameters of the algorithm. The figure are taken from the test function
    \textit{blip} of the \cref{sec:simulation-procedure}.}
  \label{fig:assess-mcmc}
\end{figure}

\section{Examples of simulations}
\label{sec:examples-simulations}

We now turn our attention to simulated examples to illustrate the performance of
mixture models. First, we use mixtures as a prior distribution on the regression
function in the univariate mean regression problem with normal errors. Of
course, the interest for mixture comes when the statistical model is more
involved. Hence, in a second time we present simulation results for the
multivariate inverse problem of CT imaging.

\subsection{Mean regression with normal errors}
\label{sec:simulation-procedure}

We present results of our algorithm on several standard test functions from the
wavelet regression litterature \citep[see][]{Marron1998}, following the
methodology from \citet{Antoniadis2001} (\textit{i.e.} Gaussian mean regression
with fixed design and unknown variance). However, it should be noticed that
mixtures are not a Bayesian new implementation of wavelet regression, and are
much more general (see for instance the next section). For each test function,
the noise variance is chosen so that the root signal-to-noise ratio is equal to
3 (a high noise level) and a simulation run was repeated 100 times with all
simulation parameters constant, excepting the noise which was regenerated. We
ran the algorithm for location-scale mixtures of Gaussians and Symmlet8, with
normal $\mathrm N(0.5,0.3)$ distribution as prior distribution on translations,
and a mixture of Gamma distributions for scales ($\mathrm{Ga}(30, 0.06)$ and
$\mathrm{Ga}(2,0.04)$ with expectation $500$ and $50$ respectively). In addition
of the core algorithm of \cref{sec:sampling-algorithm}, we also added
\begin{itemize}
  \item a Gibbs step estimation of the noise variance, with Inverse-Gamma prior
  distributon,
  \item a $\mathrm{Ga}(2,0.5)$ (with expectation $4$) prior on $\alpha$, with
  sampling of $\alpha$ done through a Gibbs update according to the method
  proposed in \citet{West1992},
  \item a Dirichlet prior on the weights of the mixture of $\mathrm{Ga}(20,0.2)$
  and $\mathrm{Ga}(2,0.1)$, with sampling of the mixture weights done through
  Gibbs sampling in a standard way,
  \item a $\mathrm{Ga}(5,10)$ (with expectation $0.5)$ prior on $T$, instead
  of normally $\mathrm{Ga(\alpha,\eta})$, which add more flexibility.
\end{itemize}
The choice of the mixture distribution as prior on scales may appear surprising,
but we found in practice that using bimodal distribution on scales substantially
improve performance of the algorithm, especially when there are few data
available and/or high noise, because in general both large and small scales
components are needed to estimate the regression function.

We ran the algorithm for $n=128$ and $n=1024$ data, and the performance is
measured by its average root mean square error, defined as the average of the
square root of the mean squared error
$n^{-1}\sum_{i=1}^{n}|\widehat{f}(x_i) - f_0(x_i)|^2$, with $\widehat{f}$
denoting the posterior mean and $f_0$ the true function. We ran on the same
dataset the Translation-Invariant with hard thresholding algorithm (TI-H) and
Symmlet8 wavelets (see \citet{Antoniadis2001}), which is one of the best
performing algorithm on this collection of test functions. We ran our algorithm
with Symmlet8 kernels to make this comparison more relevant, since the choice of
the kernel has major impact on the performance of the algorithm (see
\cref{sec:discussion} below).

\subsubsection{Alternatives}
\label{sec:alternatives}

In \cite{Wolpert2011}, authors develop a reversible-jump MCMC scheme where the
random measure is thresholded, \textit{i.e.} small jumps are removed, yielding
to a compound Poisson process approximation of the random measure, with
almost-surely a finite number of jumps, allowing numerical computations. We also
ran their algorithm with a thresholding level of $\epsilon = 0.05$ (which seems
to give the best performance), a $\mathrm{Ga}(15,1)$ prior on $\eta$, and all
other parameters being exactly the same as described in the previous section. We
use the criteria of \cref{sec:assess-conv-mark} to stop the running of the
chain.

\subsubsection{Choosing the number of particles}
\label{sec:choos-numb-part}

It is not clear how to choose the number of particles in the algorithm. In
theory, the higher is the better. In practice, however,  we recommend choosing
the number of particles according to the acceptance rate of particles weights
move in step 3 of the algorithm. We found in practice that a level of acceptance
between 20\% and 30\% is acceptable, as illustrated in
\cref{fig:monitor-particles}.

\begin{figure}[!htb]
  \centering
  \includegraphics[width=0.98\linewidth]{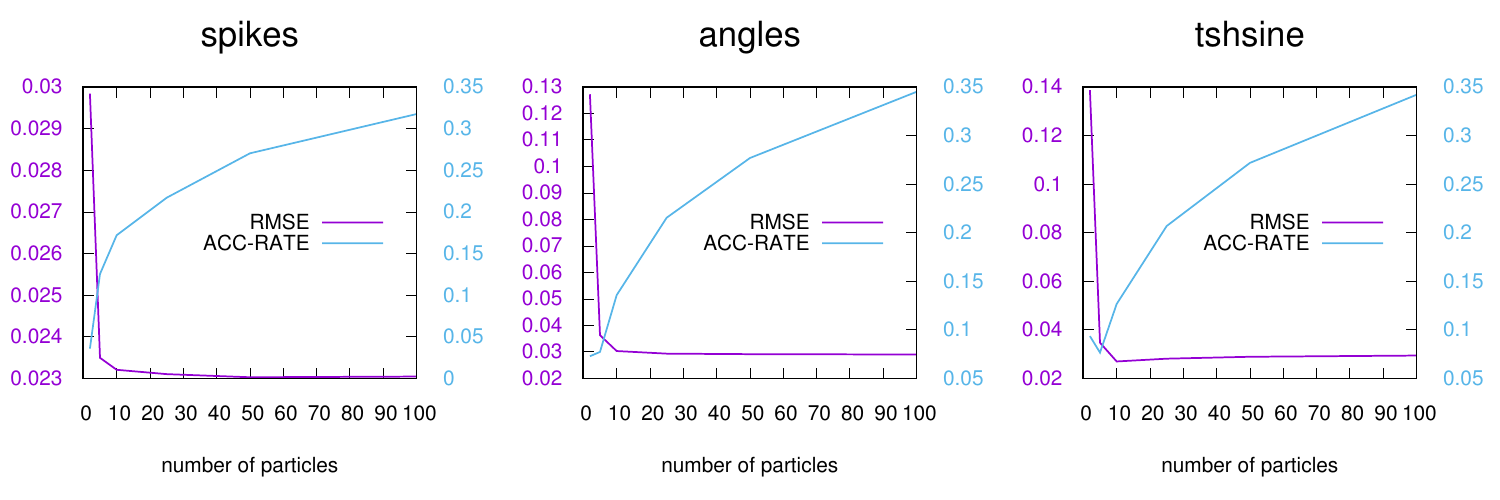}
  \caption{Mean over 100 runs of RMSE versus acceptance rate in step 3 of the
    algorithm for some typical test functions. For each signal the number of
    covariates is set to 128 and the RNSR is equal to 3.}
  \label{fig:monitor-particles}
\end{figure}

\subsubsection{Simulation results}
\label{sec:simulation-results}

In \cref{tab:tab-summary-128,tab:tab-summary-1024} we summarize the results for
location-scale mixtures of Gaussians and Symmlet8 produced by the algorithm of
\cref{sec:sampling-algorithm} and by the RJMCMC algorithm of
\citet{Wolpert2011}, with the TI-H method as reference. We used $p=150$
particles for both the datasets with $n=128$ covariates and $n=1024$ covariates,
which is a nice compromise in terms of performance and computational
cost. Regarding our algorithm and the RJMCMC algorithm, no particular effort was
made to determine the value of the fixed parameters.

\begin{table}[!htb]
  \centering
  \begin{tabular}{l|c|cc|cc}
    \toprule
    & TI-H   & \multicolumn{2}{c|}{Gibbs} & \multicolumn{2}{c}{RJMCMC}\\
    \midrule
    Function & Symm8 & Gauss & Symm8 & Gauss & Symm8 \\
    \midrule
    step      & 0.0589 & 0.0517 & 0.0551 & 0.0550 & 0.0565 \\
    wave      & 0.0319 & 0.0323 & 0.0306 & 0.0342 & 0.0370 \\
    blip      & 0.0307 & 0.0301 & 0.0316 & 0.0323 & 0.0373 \\
    blocks    & 0.0464 & 0.0343 & 0.0374 & 0.0383 & 0.0418 \\
    bumps     & 0.0285 & 0.0162 & 0.0229 & 0.0224 & 0.0345 \\
    heavisine & 0.0257 & 0.0267 & 0.0264 & 0.0280 & 0.0289 \\
    doppler   & 0.0443 & 0.0506 & 0.0418 & 0.0526 & 0.0493 \\
    angles    & 0.0293 & 0.0266 & 0.0282 & 0.0274 & 0.0305 \\
    parabolas & 0.0344 & 0.0301 & 0.0307 & 0.0312 & 0.0396 \\
    tshsine   & 0.0255 & 0.0285 & 0.0277 & 0.0291 & 0.0339 \\
    spikes    & 0.0237 & 0.0178 & 0.0207 & 0.0199 & 0.0218 \\
    corner    & 0.0177 & 0.0171 & 0.0170 & 0.0182 & 0.0255 \\
    \bottomrule
  \end{tabular}
  \caption{Summary of root mean squared errors of different algorithms for
    $n=128$ covariates and a root signal to noise ratio of $3$.}
  \label{tab:tab-summary-128}
\end{table}

\begin{table}[!htb]
  \centering
  \begin{tabular}{l|c|cc|cc}
    \toprule
    & TI-H   & \multicolumn{2}{c|}{Gibbs} & \multicolumn{2}{c}{RJMCMC}\\
    \midrule
    Function & Symm8 & Gauss & Symm8 & Gauss & Symm8 \\
    \midrule
    step      & 0.0276 & 0.0268 & 0.0289 & 0.0282 & 0.0300 \\
    wave      & 0.0088 & 0.0118 & 0.0108 & 0.0133 & 0.0117 \\
    blip      & 0.0148 & 0.0162 & 0.0172 & 0.0180 & 0.0183 \\
    blocks    & 0.0222 & 0.0230 & 0.0241 & 0.0247 & 0.0256 \\
    bumps     & 0.0122 & 0.0132 & 0.0182 & 0.0201 & 0.0232 \\
    heavisine & 0.0154 & 0.0134 & 0.0139 & 0.0147 & 0.0147 \\
    doppler   & 0.0180 & 0.0207 & 0.0196 & 0.0261 & 0.0225 \\
    angles    & 0.0123 & 0.0120 & 0.0123 & 0.0125 & 0.0128 \\
    parabolas & 0.0135 & 0.0124 & 0.0132 & 0.0147 & 0.0145 \\
    tshsine   & 0.0107 & 0.0109 & 0.0111 & 0.0131 & 0.0120 \\
    spikes    & 0.0110 & 0.0075 & 0.0095 & 0.0095 & 0.0103 \\
    corner    & 0.0077 & 0.0075 & 0.0081 & 0.0095 & 0.0085 \\
    \bottomrule
  \end{tabular}
  \caption{Summary of root mean squared errors of different algorithms for
    $n=1024$ covariates and a root signal to noise ratio of $3$.}
  \label{tab:tab-summary-1024}
\end{table}

Obviously the Gibbs algorithm allow for sampling the \textit{full} posterior
distribution, pemitting estimation of posterior credible bands, as illustrated
in \mbox{\cref{fig:crediblebands-gauss,fig:crediblebands-symmlet8}}, where the
credible bands were drawn retaining the $95\%$ samples with the smaller
$\ell_{2}$-distance with respect to the posterior mean estimator. Although the
algorithm samples an approximated version of the model, it is found that the
accuracy of credible bands is quite good since the true regression function
almost never comes outside the sampled $95\%$ bands, as it is visible in the
example of
\mbox{\cref{fig:crediblebands-gauss,fig:crediblebands-symmlet8}}. Despite the
algorithm efficiency, future work should be done to develop new sampling
techniques for regression with mixture models, mainly to improve computation
cost.

\begin{figure}[!p]
  \centering
  \includegraphics[scale=0.75]{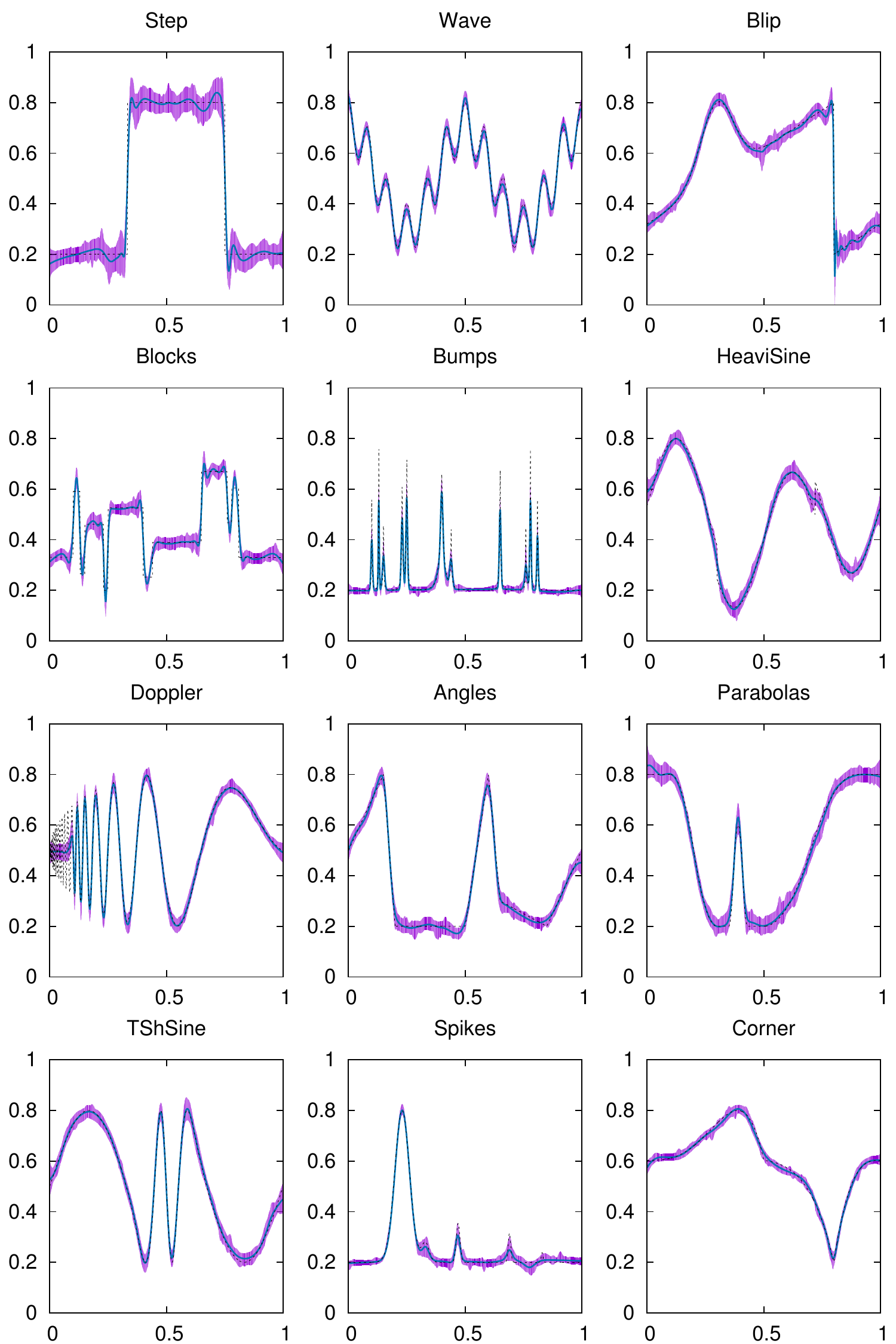}
  \caption{Example of simulation results using location-scale mixtures of
    Gaussians. The root signal-to-noise ratio is equal to 3 for sample size of
    1024 design points. The true regression function is represented with dashes,
    the mean of the sampled posterior distribution in blue and sampled 95\%
    credible bands in pink.}
  \label{fig:crediblebands-gauss}
\end{figure}

\begin{figure}[!p]
  \centering
  \includegraphics[scale=0.75]{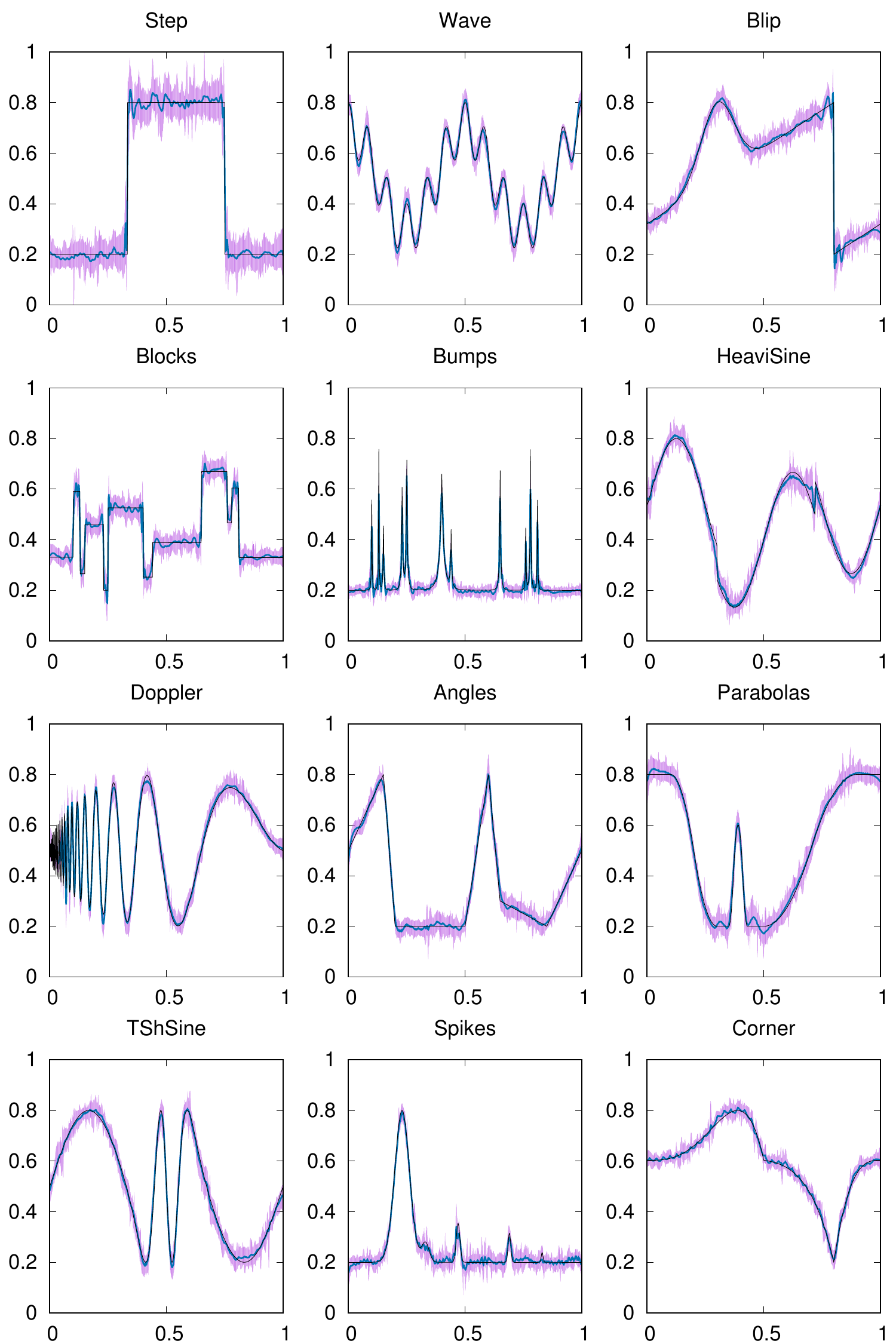}
  \caption{Example of simulation results using location-scale mixtures of
    Symmlet8. The root signal-to-noise ratio is equal to 3 for sample size of
    1024 design points. The true regression function is represented with dashes,
    the mean of the sampled posterior distribution in blue and sampled 95\%
    credible bands in pink.}
  \label{fig:crediblebands-symmlet8}
\end{figure}

\subsubsection{Discussion}
\label{sec:discussion}

Obviously, the computation cost for our algorithm is high compared to TI-H, or
any other classical wavelet thresholding method, even considering that it can
intrinsically compute credible bands. But, as mentioned in
\citet{Antoniadis2001}, the choice of the kernel is crucial to the performance
of estimators. The attractiveness of mixtures then comes because we are not
restricted to location-scale or location-modulation kernels, and almost any
function is acceptable as a kernel, which is not the case for most regression
methods. Moreover, there is no requirements on how the data are spread, which
makes the method interesting in inverse problems, such as in the next section.

\subsection{Multivariate inverse problem example}
\label{sec:2d-example}

Many medical imaging modalities, such as X-ray computed tomography imaging (CT),
can be described mathematically as collecting data in a Radon transform
domain. The process of inverting the Radon transform to form an image can be
unstable when the data collected contain noise, so that the inversion needs to
be regularized in some way. Here we model the image of interest as a measurable
function $f : \mathbb R^2 \rightarrow \mathbb R$, and we propose to use a
location-scale mixtures of Gaussians to regularize the inversion of the Radon
transform.

More precisely, the Radon transform
$R_f : \mathbb R^+ \times [0,\pi] \rightarrow \mathbb R$ of $f$ is such that
$R_f(r,\theta) = \int_{-\infty}^{+\infty}f(r\cos\theta - t\sin \theta, r\sin
\theta + \cos\theta)\,dt$. Then we consider the following model. Let
$n,m \geq 1$. Assuming that the image is supported on $[-1,1]^2$ we let
$r_1,\dots,r_n$ equidistributed in $[-\sqrt{2},\sqrt{2}]$ and
$\theta_1,\dots,\theta_m$ equidistributed in $[0,\pi]$. Then,
\begin{gather*}
  Y_{nm} \sim \mathrm{N}(R_f(x_n,\theta_m),\sigma^2)\quad \forall n,m\\
  f \sim \Pi,
\end{gather*}
where $\Pi$ is a symmetric Gamma process location-scale mixture with base
measure $\alpha F_A \times F_{\mu}$ on $\mathcal{E}\times \mathbb R^2$,
$\alpha > 0$, and scale parameter $\eta > 0$. In the sequel, we use a normal
distribution with mean zero and covariance matrix $\mathrm{diag}(\tau,\tau)$ as
distribution for $F_{\mu}$. Regarding $F_A$, the choice is more delicate; we
choose a prior distribution over the set of \textit{shearlet-type} matrices of
the form
\begin{equation*}
  \begin{pmatrix}
    1 & s\\
    0 & 1
  \end{pmatrix}
  \begin{pmatrix}
    a & 0\\
    0 & \sqrt{a}
  \end{pmatrix},
\end{equation*}
where we set a $\mathrm{N}(1,\sigma_a^2)$ distribution over the coefficient $a$
and $\mathrm{N}(0,\sigma_s^2)$ over the coefficient $s$. This type of prior
distribution for $F_A$ is particularly convenient for capturing anisotropic
features such as edges in images \citep{EasleyColonnaLabate2009}.

We ran our algorithm for $n=256$ and $m=128$ ($32768$ observations, a small
amount), using the Shepp and Logan phantom as original image
\citep{SheppLogan1974}. The variance of the noise is $\sigma^2 = 0.1$, whereas
the image take value between $0$ and $2$. Both the original image and the
reconstruction are visible in \cref{fig:ct-imaging}. Finally, we should mention
that the choice of the Gaussian kernel for the mixture is convenient since it
allows to compute the likelihood analytically. However, from a practical side, a
full implementation of the algorithm with the intention of reconstructing CT
images may benefit from using a different kernel.
\begin{figure}[!htb]
  \centering
  \includegraphics[width=0.75\linewidth]{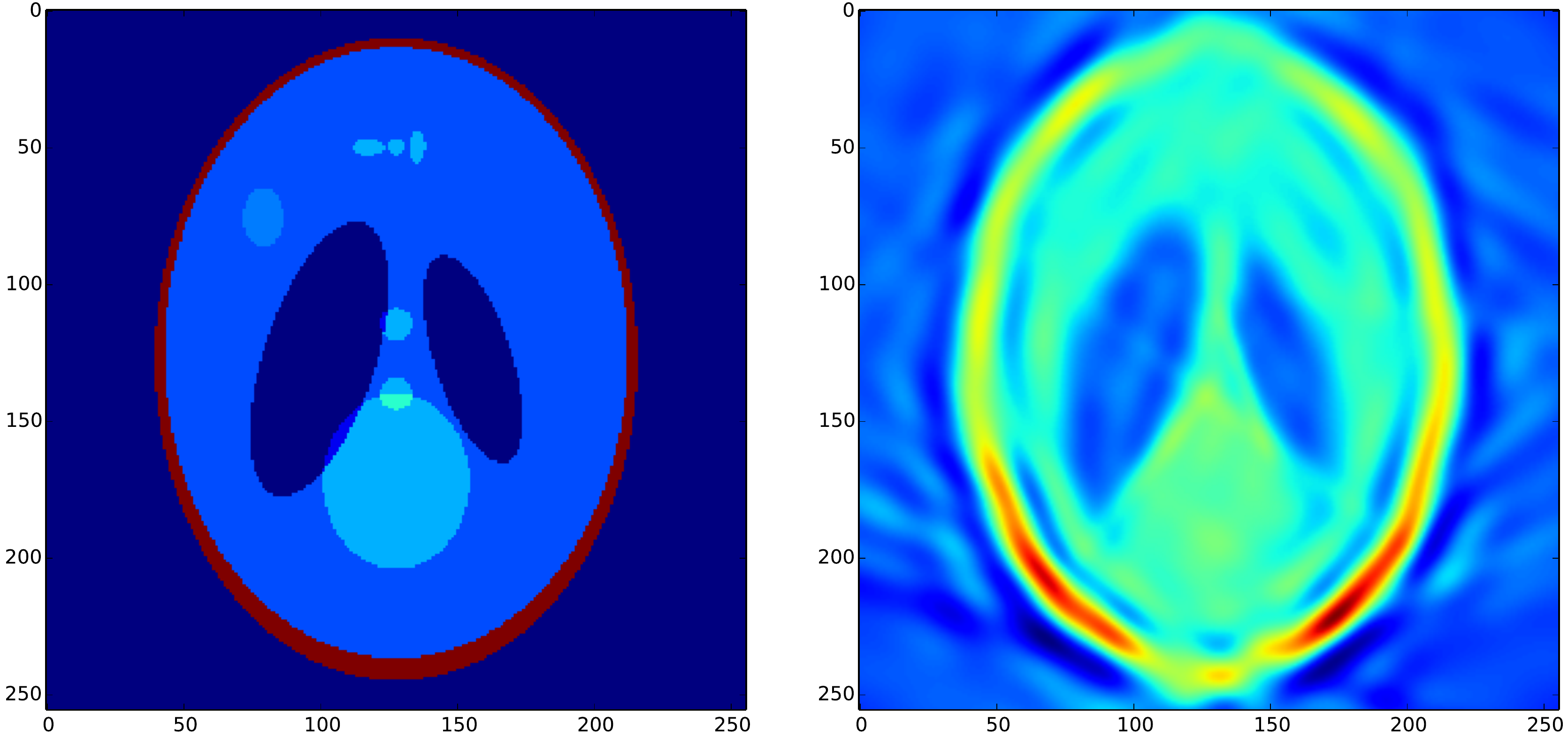}
  \caption{Simulation of X-ray computed tomography imaging using symmetric Gamma
  process location-scale mixture of Gaussians. On the left: the original
  image. On the right: the reconstructed image from $32768$ observations of the
  Radon transform of the original image in a Gaussian noise.}
  \label{fig:ct-imaging}
\end{figure}

\section{Rates of convergence}
\label{sec:rates-convergence-loc-scale}

In this section, we investigate posterior convergence rates in fixed design
Gaussian regression for both symmetric Gamma location-scale mixtures and
symmetric Gamma location-modulation mixtures.

\subsection{Notations}
\label{sec:notations}

In the sequel we use repeatedly the following notations.

\begin{itemize}
  \item The conventional multi-index notation, for all
  $\alpha = (\alpha_1,\dots,\alpha_d)\in\integers^d$ and all
  $z = (z_1,\dots,z_d) \in\reals^d$ we write
  $|\alpha| := \alpha_1+\dots+\alpha_d$, $\alpha! := \alpha_1!\dots\alpha_d!$,
  and $z^{\alpha} := z_1^{\alpha_1}\dots z_d^{\alpha_d}$. Moreover, for all
  $f:\reals^d\rightarrow\reals$ with continuous $k$-th order partial
  derivatives at $x\in\reals^d$ we write
  \begin{align*}
    D^{\alpha}f(x) := \frac{\partial^{|\alpha|}f}{\partial
    z_1^{\alpha_1}\dots\partial z_n^{\alpha_d}}(x),\quad |\alpha| \leq k.
  \end{align*}
  \item Let $\Omega$ be an open subset of $\reals^d$ and $\overline{\Omega}$ be
  the closure of $\Omega$. For any $\beta > 0$, we define
  $\mathcal{C}^{\beta}(\overline{\Omega})$, the Hölder space on
  $\overline{\Omega}$, as the set of all functions on $\overline{\Omega}$ such
  that
  $\|f\|_{\mathcal{C}^{\beta}} := \max_{|\alpha| \leq k}\sup_{x\in
    \Omega}|D^{\alpha}f(x)| + \max_{|\alpha| = k}\sup_{x\ne y \in
    \Omega}|D^{\alpha}f(x) - D^{\alpha}f(y)|/|x - y|^{\beta - k}$ is finite,
  where $k$ is the largest integer strictly smaller than $\beta$.
  \item We denote by $|\cdot|_d$ the standard euclidean norm on $\reals^d$, and,
  for any $x,y\in\reals^d$, $xy$ is the standard inner product. For any
  $d\times d$ matrix $A$ with real eigenvalues, we denote
  $\lambda_1(A) \geq \dots \geq \lambda_d(A)$ its eigenvalues in decreasing
  order, $\|A\| := \sup_{x\ne 0}|A x|_d/|x|_d$ its spectral norm, and
  $\|A\|_{\max} := \max_{i,j}|A_{ij}|$, where $A_{ij}$ are the entries of $A$.
  \item Given a signed measure $\mu$ on a measurable space
  $(\mathcal{X},\mathcal{A})$, we let $\mu^+$ and $\mu^-$ denote respectively
  the positive and negative part of the Jordan decomposition of $\mu$. Also,
  $|\mu| = \mu^{+} + \mu^-$ denote the total variation measure of $\mu$.
  \item Inequalities up to a generic constant are denoted by the symbols
  $\lesssim$ and $\gtrsim$.
\end{itemize}

\subsection{The model}
\label{sec:the-model}

We consider the problem of a random response $Y$ corresponding to a
deterministic covariate vector $x$ taking values in $[-S,S]^d$ for some $S >
0$. We aim at estimating the regression function $f:[-S,S]^d\rightarrow \reals$
such that $f(x_i) = \Esp Y_i$, based on independent observations of $Y$. More
precisely, the nonparametric regression model we consider is the following,
\begin{align*}
  Y_i|\epsilon_i
  &= f(x_i) + \epsilon_i,\quad i=1,\dots,n,\\
  \epsilon_1,\dots,\epsilon_n|\sigma^2
  &\overset{\mathrm{i.i.d}}{\sim}
    \mathcal{N}(0,\sigma^2),\quad\textrm{independently of $(f,\sigma)$},\\
  (f,\sigma) &\sim \Pi,
\end{align*}
with $\Pi$ the distribution on an abstract space $\Theta$, given by
$\sigma \sim P^{\sigma}$ independently of $f$ drawn from the distribution of a
symmetric Gamma process mixture.

\subsection{A general result}
\label{sec:general-result}

Let $P_{\theta,i}$ denote the distribution of of $Y_i$ under the parameter
$\theta = (f,\sigma)$, $P_{\theta}^n$ denote the joint distribution of
$(Y_1,\dots,Y_n)$, $P_{\theta}^{\infty}$ the distribution of the infinite
sequence $(Y_1,\dots,Y_{\infty})$, and
$\|f\|_{2,n}^2 := n^{-1}\sum_{i=1}^n|f(x_i)|^2$. Let define the distance
$\rho_n(\theta_0,\theta_1) := \|f - f_0\|_{2,n} + |\log \sigma_0 - \log
\sigma_1|$. For the regression method based on $\Pi$, we say that its posterior
convergence rate at $\theta_0$ in the metric $\rho_n$ is $\epsilon_n$ if there
is $M < +\infty$ such that
\begin{align}
  \label{eq:9}
  \lim_{n\rightarrow \infty}\Pi\left(\Set*{ \theta \in
  \Theta \given \rho_n(\theta,\theta_0) > M \epsilon_n } | Y_1,\dots,Y_n \right)
  = 0\quad P_{\theta_0}^{\infty}\textrm{-a.s}.
\end{align}

Most of the approach to rates of convergence rely on idea coming from density
mixtures models \citep{Ghosal2000,Shen2013, Canale2013}. Indeed, we prove that
\cref{eq:9} hold by verifying a set of sufficient conditions established in
\cref{thm:1}. For $\epsilon > 0$ and any subset $A$ of a metric space equipped
with metric $\rho$, let $N(\epsilon, A, \rho)$ denote the $\epsilon$-covering
number of $A$, \textit{i.e.} the smallest number of balls of radius $\epsilon$
needed to cover $A$. Also, for all $i=1,\dots,n$, define
$K_i(\theta_0,\theta) := \int (\log dP_{\theta_0,i}/ dP_{\theta,i})\,
dP_{\theta_0,i}$ and
$V_{2,i}(\theta_0,\theta) := \int (\log dP_{\theta_0,i}/dP_{\theta,i} -
K_i(\theta_0,\theta))^2\, dP_{\theta_0,i}$, and let
\begin{align*}
  K_n(\theta_0,\epsilon)
  := \Set*{ \theta \given \frac 1n \sum_{i=1}^n K_i(\theta_0,\theta) \leq
  \epsilon^2,\ \frac 1n \sum_{i=1}^n V_{2,i}(\theta_0,\theta) \leq \epsilon^2},
\end{align*}
be the Kullback-Leibler ball of size $\epsilon$ around
$\theta_0:= (f_0,\sigma_0)$. \Cref{thm:1} is the analogue of theorem~5 in
\citet{GhosalVanDerVaart2007} for the Gaussian mean regression with fixed
design~; the major difference reside on constructing suitable test functions,
and extra cares have to taken regarding the fact that observations are not
i.i.d. The proof of \cref{thm:1} is given in \cref{sec:proofs-crefs-result}.

\begin{theorem}
  \label{thm:1}
  Let $K := 3(32 \vee 4\sigma_0^2)^{-1}$, and $\epsilon_n \rightarrow 0$ with
  $n\epsilon_n^2 \rightarrow \infty$.  Suppose that $\Theta_n \subset \Theta$ is
  such that $\Pi(\Theta_n^c) \lesssim \e^{-3 n \epsilon_n^2}$ for $n$ large
  enough. Assume that $\Theta_n \subseteq \cup_j \Theta_{n,j}$ is such that for
  some $M > 0$,
  \begin{gather*}
    \textstyle \lim_{n} \sum_{j} \sqrt{N(M\epsilon_n,\Theta_{n,j},\rho_n)}
    \sqrt{\Pi(\Theta_{n,j})} \e^{-(KM^2 - 2) n\epsilon_n^2} = 0,\\
    \Pi(K_n(\theta_0,\epsilon_n)) \gtrsim \e^{- n \epsilon_n^2}.
  \end{gather*}
  Then
  $\Pi(\theta \in \Theta\, :\, \rho_n(\theta_0,\theta) > 11
  M\epsilon_n|Y_1,\dots,Y_n) \rightarrow 0$ in $P_{\theta_0}^n$-probability.
\end{theorem}

\subsection{Supplementary assumptions}
\label{sec:suppl-assumpt}

In order to derive rates of convergence (and only for this) we make
supplementary assumptions on the choice of the mother function $g$ and of the
base measure $\alpha F$.

\subsubsection{Location-scale mixtures}
\label{sec:locat-scale-mixt-1}

We restrict our discussion to priors for which the following conditions are
verified. We assume that

\begin{itemize}
  \item $g : \mathbb R^d \rightarrow \mathbb R$ is a non zero Schwartz function
  such that $|g(x)| \lesssim \exp(-C_0|x|_d^{\tau})$ for some $C_0, \tau >
  0$. We assume that there is $0 \leq \gamma < 1$ such that
  $\sup_{|\alpha|=k}|D^{\alpha}g(0)| \lesssim \exp(\gamma k \log k)$ for all $k$
  large enough ; this last assumption is not obvious, it is for example met with
  $\gamma = 1/2$ if $g$ is a multivariate Gaussian (see \cref{pro:7} in
  appendix).
  \item $\alpha F := \alpha F_A \times F_{\mu}$, where $F_A$ is a probability
  measure on $\mathcal{E}_s$, the space of symmetric positive definite
  $d\times d$ reals matrices, and $F_{\mu}$ a probability measure on
  $[-2S,2S]^d$. We also assume that there exist positive constants $\kappa > 0$,
  $\kappa^{*} > d(d-1)$, $a_1,\dots,a_5$, $b_1,\dots,b_{6}$, $C_1,\dots,C_{3}$
  such that for any $0 < s_1 \leq \dots\leq s_d$, any $z_0 \in [-2S,2S]^d$, all
  $t\in (0,1)$ and all $x>0$ sufficiently large
  \begin{gather}
    \label{eq:16}
    F_{\mu}(z\ :\ |z - z_0| \leq t) \geq b_1t^{a_1},\\
    \label{eq:19}
    F_A(A\ :\ \lambda_d(A^{-1}) \geq x) \leq b_2\exp(-C_2 x^{a_2}),\\
    \label{eq:18}
    F_A(A\ : \ \lambda_1(A^{-1}) < 1/x) \leq b_3x^{-a_3},\\
    \label{eq:14}
    F_A\left( A\ :\ s_j < \lambda_j(A^{-1}) < s_j(1+ t),\ 1\leq j \leq d \right)
    \geq
    b_4s_d^{a_4}t^{a_5}\exp(-C_3s_d^{\kappa/2}),\\
    \label{eq:31}
    F_A(A\ : \ \lambda_1(A)/\lambda_d(A) > x) \leq b_6x^{-\kappa^{*}}.
  \end{gather}
  \Cref{eq:19,eq:18,eq:14} are classical and are met for instance with
  $\kappa = 2$ if $F_A$ is the inverse-Wishart distribution
  \citep[lemma~1]{Shen2013}. For a thorough discussion about \cref{eq:31} we
  refer to \citet{Canale2013} and references therein.
  \item $P^{\sigma}$ is a probability distribution on $(0,\infty)$. We also
  assume that there are positive constants $a_7,a_8,a_9$, $b_7,b_8$, $C_8$, and
  $b_9$ eventually depending on $\sigma_0 > 0$, such that for all $t\in (0,1)$
  \begin{gather}
    \label{eq:29}
    P^{\sigma}(\sigma\ : \ \sigma > x) \leq b_7 x^{-a_7},\\
    \label{eq:27}
    P^{\sigma}(\sigma\ : \ \sigma \leq 1/x) \leq b_8\exp(-C_8 x^{a_8}),\\
    \label{eq:8b}
    P^{\sigma}\left(\sigma \ :\ \sigma_0 \leq \sigma \leq \sigma_0(1 + t)\right)
    \geq b_9 t^{a_9}.
  \end{gather}
\end{itemize}

\subsubsection{Location-modulation mixtures}
\label{sec:locat-modul-mixt-1}

We restrict our discussion to priors for which the following conditions are
verified. We assume that

\begin{itemize}
  \item $g : \mathbb R^d \rightarrow \mathbb R$ is a non zero Schwartz function
  such that $g(x) \geq 0$ for all $x\in \mathbb R^d$ and
  $|g(x)| \lesssim \exp(-C_0|x|_d^{\tau})$ for some $C_0 > 0$ and $\tau > 1$. We
  assume that there is a set $E \subseteq [-\pi,\pi]^d$ with strictly positive
  Lebesgue measure and a constant $C > 0$ such that $g(x) \geq C$ on $E$. We
  also assume that there is $0 \leq \gamma < 1$ such that
  $\sup_{|\alpha|=k}|D^{\alpha}g(0)| \lesssim \exp(\gamma k \log k)$ for all $k$
  large enough. As in the previous section, these assumptions are met for the
  multivariate Gaussian with $E=[-\pi,\pi]^d$, $\gamma = 1/2$ and $\tau = 2$
  (see \cref{pro:7} in appendix).
  \item $\alpha F := \alpha F_\xi \times F_{\mu} \times F_{\phi}$, where
  $F_{\xi}$ is a probability measure on $\mathbb R^d$, $F_{\mu}$ a probability
  measure on $[-2S,2S]^d$, and $F_{\phi}$ a probability measure on
  $[0,\pi/2]$. For all $t\in (0,1)$ and all $z_0\in [-2S,2S]^d$ we assume that
  $F_{\mu}$ satisfies \cref{eq:16}. We assume that there is positive constants
  $a_{10},b_{10}$ such that for all $t\in (0,1)$ and all $\phi_0 \in [0,\pi/2]$
  we assume that
  $F_{\phi}(\phi\ :\ |\phi - \phi_0| \leq t) \geq b_{10}t^{a_{10}}$. We also
  assume that there exist positive constants $\eta > (d-1)/2$,
  $a_{12},a_{13},b_{11},b_{12}$ such that for all $t\in (0,1)$, all
  $\xi_0 \in \mathbb R^d$ and for all $x > 0$
  \begin{gather}
    \label{eq:11}
    F_{\xi}(\xi\ :\ |\xi|_d \geq x) \leq b_{11}(1 + x)^{-2(\eta + 1)}\\
    \label{eq:15}
    F_{\xi}(\xi\ : \ |\xi - \xi_0|_d \leq t) \geq
    b_{12}|\xi_0|_d^{-a_{12}}t^{a_{13}}.
  \end{gather}
  \item $P^{\sigma}$ satisfies the same assumptions of \cref{eq:29,eq:27,eq:8b}.
\end{itemize}

\subsection{Results}
\label{sec:results}

\Cref{thm:1} serves as a starting point for proving rates of contraction for
symmetric Gamma process location-scale and location-modulation mixtures in the
model of \cref{sec:the-model}. The proofs of the next theorems resemble to
\citet{DeJongeVanZanten2010}, but, they consider only a location mixture with
locations taken on a lattice, allowing for a very specific construction of the
sets $\Theta_n$. Here, we do not assume that locations are spread over a
lattice, which makes the construction of $\Theta_n$ more involved. Our
construction is inspired from \citet{Shen2013} for Dirichlet processes mixtures,
but adapted to symmetric Gamma processes (indeed, the same construction should
work for many Lévy processes). Also, \cref{thm:1} allows for partitioning
$\Theta_n$ onto slices $\Theta_{n,j}$, a step which is unnecessary for location
mixtures \citep{DeJongeVanZanten2010,Shen2013}, but yields to better rates and
weaker assumptions on the prior when dealing with location-scale
\citep{Canale2013} and location-modulation mixtures.

Regarding the model of \cref{sec:the-model}, with deterministic covariates
$x_1,\dots,x_n$ arbitrary spread in $[-S,S]^d$, we have the following theorem
for location-scale mixtures. We notice that unlike \citet{DeJongeVanZanten2010},
we do not assume that the covariates are spread on a strictly smaller set than
$[-S,S]^d$, \textit{i.e.} the support of the covariates and the domain of the
regression function are the same.
\begin{theorem}
  \label{thm:3}
  Let $\zeta = 1 \vee 2/(\tau - \gamma \tau)$. Suppose that
  $f_0 \in \mathcal{C}^{\beta}[-S,S]^d$ for some $S > 0$. Under the assumptions
  of \cref{sec:suppl-assumpt}, the \cref{eq:9} holds for the location-scale
  prior with
  $\epsilon_n^2 = n^{-2\beta/(2\beta + d + \kappa/2)}(\log n)^{2\beta d(\zeta -
    1)/(2\beta + d + \kappa/2)}$.
\end{theorem}

\Cref{thm:3} gives a rate of contraction analogous to the rates found in
\citet{Canale2013}, that is to say, suboptimal with respect to the frequentist
minimax rate of convergence. Indeed, if one use an Inverse-Wishart distribution
for $F_A$, then $\kappa=2$; we can achieve $\kappa =1$ with a distribution
supported on diagonal matrices which assign square of inverse gamma random
variables to non-null element of the matrix. Obviously, the choice of $F_A$
matters since it has a direct influence on the rates of contraction of the
posterior. Also notice that the rates depends on $\kappa/2$, which is slightly
better than the $\kappa$ dependency found \citet{Canale2013}. The reason is
relatively artificial, since this follows from the fact that we put a prior on
dilation matrices of the mixture, whereas they set a prior on square of dilation
matrices (covariance matrices).

Location-modulation mixtures were never considered before, because they are not
satisfactory for estimating a density. In comparison with location-scale
mixtures, the major difference in proving contraction rates rely on
approximating sufficiently well the true regression function. We use a new
approximating scheme, based on standard of Fourier series analysis, yielding the
following theorem.
\begin{theorem}
  \label{thm:4}
  Suppose that $f_0 \in \mathcal{C}^{\beta}[-S,S]^d$ for some $S > 0$. Under the
  assumptions of \cref{sec:suppl-assumpt}, the \cref{eq:9} holds for the
  location-modulation prior with
  \begin{equation*}
    \epsilon_{n}^2 = n^{-2\beta/(2\beta +d)}(\log n)^{2\beta(2d+1)/(2\beta+d)}.
  \end{equation*}
\end{theorem}

Although it was not surprising that location-scale mixtures yield suboptimal
rates of convergence, we would have expected that location-modulation mixtures
could be suboptimal too, which is not the case (up to a power of $\log n$
factor). Moreover, location-modulation mixtures seem less stiff than location
mixtures \citep{Shen2013}, hence they might be interesting to consider in
regression.

Finally, it should be mentioned that all the rates here are adaptive with
respect to $\beta > 0$; that is location-scale and location-modulation mixtures
achieve these rates simultaneously for all $\beta > 0$.

\section{Proofs of \cref{sec:algorithm}}
\label{sec:proof-convergence-series}

\subsection{Preliminaries on convergence of signed random measures}
\label{sec:prel-conv-sign}

It is well known for random (non-negative) measures that it is enough to show
weak convergence of finite dimensional distributions on a semiring of bounded
sets generating $\mathcal{A}$ to prove vague convergence of the distribution,
see for instance \citet[Theorem~4.2]{Kallenberg1983} or
\citet[Theorem~11.1.VII]{Daley2007}. This fact remains true for random signed
measures, but not in a obvious way. Indeed, it is well known that the vague
topology is not metrizable on $\mathcal{M}(\spx)$, even if $\spx$ is Polish (for
example, see Remark~1.2 in \citet{DelBarrio2007}), making the vague topology
nasty to handle on $\mathcal{M}(\spx)$. In particular, it is not as direct as in
the case of non-negative measures to prove that the $\sigma$-algebra generated
by the sets
$\Set{ \Set{\mu \in \mathcal{M}(\spx) \given \mu(B)\in A} \given A \in \mathscr
  B(\reals),\ B \in \mathcal{R}}$,
where $\mathcal{R}$ is a ring of bounded sets generating $\mathcal{A}$,
coincides with the Borel $\sigma$-algebra of $\mathcal{M}(\spx)$, given the
topology of vague convergence. However, once this last fact is proved,
everything in the proof of \citet[Theorem~4.2]{Kallenberg1983} remains valid for
signed random measures.

\par Surprisingly, there is not so much literature on vague convergence of
signed random measures, and as our knowledge, the only reference available on
this subject is \citet{Jacob1995}. We state here the result of interest for us,
with only a sketch of the proof, as the details can be found in the original
article.

\begin{lemma}
  \label{lem:3}
  Let $\mathcal{R} \subset \mathcal{A}$ denote the ring of bounded Borel sets of
  $\spx$. Then the Borel $\sigma$-algebra of $\mathcal{M}(\spx)$ (given the
  weak-* topology) coincides with the $\sigma$-algebra generated by the sets
  $\Set{ \Set{\mu \in \mathcal{M} \given \mu(B)\in A} \given A \in \mathscr
    B(\reals),\ B \in \mathcal{R}}$
  and also
  $\Set{ \Set{\mu \in \mathcal{M} \given \mu(f)\in A} \given A \in \mathscr
    B(\reals),\ f \in \mathtt C_c(\spx)}$.
\end{lemma}
\begin{proof}[Sketch of proof]
  First, we shall prove that
  $\mathcal{S} := \sigma\Set{ \Set{\mu \in \mathcal{M} \given \mu(B)\in A}
    \given A \in \mathscr B(\reals),\ B \in \mathcal{R}} = \sigma \Set{ \Set{\mu
      \in \mathcal{M} \given \mu(f)\in A} \given A \in \mathscr B(\reals),\ f
    \in \mathtt C_c(\spx)}$.
  Using the Hahn-Jordan decomposition of signed measures, this is a
  straightforward adaptation of \citet[Lemma~1.4]{Kallenberg1983}.

  Also, the argument of \citet[Lemma~4.1]{Kallenberg1983} for proving
  $\mathcal{S} \subset \mathscr{B}(\mathcal{M})$ remains valid here, but the
  converse inclusion is not as direct.
  Let $\mathcal{M}^+ \subset \mathcal{M}$ denote the cone of non-negative
  measures, and endow $\mathcal{M}^+$ with the topology $\mathcal{T}^+_{v}$ of
  vague convergence (\textit{i.e.} $\mu_n$ converges to $\mu$ if
  $\mu_n(f) \rightarrow \mu(f)$ for any $f\in \mathtt C_c^+$) and corresponding
  Borel $\sigma$-algebra $\mathscr B(\mathcal{M}^+)$. We denote $\mathcal{S}^+$
  the trace of $\mathcal{S}$ over $\mathcal{M}^+$.
  Hence, it suffices to prove that
  \begin{enumerate}
    \item\label{item:1} $\mathcal{S}^+ = \mathscr B(\mathcal{M}^+)$,
    \item\label{item:2}
    $P : (\mathcal{M},\mathcal{S}) \rightarrow (\mathcal{M}^+\times
    \mathcal{M}^+, \mathcal{S}^+\times \mathcal{S}^+)$,
    such that $P(\mu) := (\mu^+,\mu^-)$, is measurable,
    \item\label{item:3}
    $R : (\mathcal{M}^+\times \mathcal{M}^+, \mathscr{B}(\mathcal{M}^+)\times
    \mathscr{B}(\mathcal{M}^+)) \rightarrow (\mathcal{M},
    \mathscr{B}(\mathcal{M}))$,
    such that $R(\mu,\nu) := \mu - \nu$, is measurable.
  \end{enumerate}
  These 3 conditions imply that
  $R\circ P : (\mathcal{M},\mathcal{S}) \rightarrow
  (\mathcal{M},\mathscr{B}(\mathcal{M}))$
  is $\mathcal{S}/\mathscr{B}(\mathcal{M})$-measurable, and since $R\circ P$ is
  just the identity mapping, this implies
  $\mathscr{B}(\mathcal{M}) \subset \mathcal{S}$, as required.
\end{proof}

\subsection{Proofs}
\label{sec:proofs}

\begin{proof}[Proof of \cref{thm:2}]
  In the whole proof, we use the Pochhammer symbols $x^{(n)}$ and $(x)_{n}$ for
  respectively the $n$th power of the increasing factorial of $x$, and the $n$th
  power of the decreasing factorial of $x$. Once we took care of subtlety coming
  with \cref{sec:prel-conv-sign}, the rest of the proof is identical to the
  proof of Proposition~A.1 in \citet{Favaro2012}, which we resume here for the
  sake of completeness. According to \cref{sec:prel-conv-sign} it is enough to
  check that
  \begin{align}
    \label{eq:5b}
    (Q_p(A_1),\dots,Q_p(A_{k})) \overset{d}{\longrightarrow}
    (Q(A_1),\dots,Q(A_{k})),
  \end{align}
  for any collection of disjoints bounded measurable sets
  $A_1,\dots,A_{k} \in \mathcal{A}$, where $Q$ is a symmetric Gamma random
  measure with parameters $\alpha F(\cdot), \eta$. Oviously, for any vector
  $(v_1,\dots,v_{k}) \in \reals^{k}$ the random variable
  $v_1Q(A_1)+\dots+v_{k}Q(A_{k})$ has symmetric Gamma distribution, and hence is
  determined by its moments (because of \cref{pro:1}), by
  \citet[Theorem~30.2]{Billingsley2008} the \cref{eq:5b} holds if
  \begin{align}
    \label{eq:2c}
    \Esp\left[Q_p(A_1)^{r_1}\dots Q_p(A_{k})^{r_{k}}\right]
    &\longrightarrow
      \Esp\left[Q(A_1)^{r_1}\dots Q(A_{k})^{r_{k}}\right]
  \end{align}
  holds for any disjoints bounded measurable sets
  $A_1,\dots,A_{k}\in \mathcal{A}$ and any positive integers
  $r_1,\dots,r_{k}$. From now, for all collection of measurable sets
  $A_1,\dots,A_{k}\in \mathcal{A}$, we set
  $A^c := \spx \backslash \cup_{i=1}^{k}A_i$. We recall that if
  $\Set{X_i \in \spx \given i\leq 1 \leq p}$ is a P{\'o}lya urn sequence with
  base distribution $\alpha F(\cdot)$, and $A_1,\dots,A_{k} \in \mathcal{A}$ are
  disjoints, then
  \begin{multline*}
    P(\#\Set{i \given X_i \in A_1} = j_1,\dots,\#\Set{i \given X_i \in A_{k}}
    = j_k)\\
    = \binom{p}{j_1\dots j_k} \frac{(\alpha F(A_1))^{(j_1)}\dots(\alpha
      F(A_k))^{(j_k)} (\alpha F(A^c))^{(p - \sum_{i=1}^kj_i)}}
    {(p-\sum_{i=1}^{k}j_i)!\,\alpha^{(p)}},
  \end{multline*}
  where $(j_1,\dots,j_k) \in \mathcal{E}_{k,p}$, with
  $\mathcal E_{k,p} := \Set{(j_1,\dots, j_k) \in \Set{0,\dots,p}^k \given
    \sum_{i=1}^kj_i \leq p}$.  It is straightforward to show that both the lhs
  and the rhs of \cref{eq:2c} are null whenever one of the $r_i$'s is
  odd. Therefore we shall only consider \cref{eq:2c} for even exponents. We
  deduce from \cref{pro:1} that for any disjoints bounded measurable sets
  $A_1,\dots,A_k \in \mathcal{A}$ and any positive integers $r_1,\dots,r_k$,
  \begin{multline*}
    \Esp\left[Q_p(A_1)^{2r_1}\dots Q_p(A_k)^{2r_k}\right] =
    \alpha^{(r_1+\dots+r_k)}
    \left(\prod_{i=1}^k\frac{(2r_i)!/r_i!}{(\sqrt{\eta})^{2r_i}p^{r_i}}
    \right)\\
    \times \sum_{(j_1,\dots,j_k)\in \mathcal{E}_{k,p}} \binom{p}{j_1\dots j_k}
    \frac{(\alpha F(A_1))^{(j_1)}\dots (\alpha F(A_k))^{(j_k)}(\alpha
      F(A^c))^{(p-\sum_{i=1}^kj_i)}}
    {(p-\sum_{i=1}^{k}j_i)!\, \alpha^{(p)}}\\
    \times (j_1)^{(r_1)}\dots (j_k)^{(r_k)}.
  \end{multline*}
  Introducing $s(\cdot,\cdot)$ and $S(\cdot,\cdot)$ are the Stirling numbers of
  the first and second kind, we can mimic \citet[Appendix~A.1]{Favaro2012} to
  find that
  \begin{multline*}
    \Esp\left[Q_p(A_1)^{2r_1}\dots Q_p(A_k)^{2r_k}\right] =
    \alpha^{(r_1+\dots+r_k)}
    \left(\prod_{i=1}^k\frac{(2r_i)!/r_i!}{(\sqrt{\eta})^{2r_i}p^{r_i}}
    \right)\\
    \times \sum_{m_1=0}^{r_1}|s(r_1,m_1)|\sum_{s_1=0}^{m_1}S(m_1,s_1)\dots
    \sum_{m_k=0}^{r_k}|s(r_k,m_k)|\sum_{s_k=0}^{m_k}S(m_k,s_k)\\
    \times \frac
    {(\alpha F(A_1))^{(s_1)}\dots (\alpha F(A_k))^{s_k}}
    {\alpha^{(s_1+\dots+s_k)}}
    (p)_{s_1+\dots+s_k}.
  \end{multline*}
  Therefore, we conclude that,
  \begin{align*}
    \lim_{p\rightarrow\infty} \Esp\left[Q_p(A_1)^{2r_1}\dots
    Q_p(A_k)^{2r_k}\right]
    &=
      \prod_{i=1}^k\left(\frac{(2r_i)!}{r_i!}
      \frac{(\alpha F(A_i))^{(r_i)}}{(\sqrt{\eta})^{2r_i}} \right).
      \qedhere
  \end{align*}
\end{proof}

\begin{proof}[Proof of~\cref{pro:3}]
  We can assume that all the $Q_p$ and $Q$ are defined on the same probability
  space $(\Omega,\mathcal{F},\mathbb P)$. The proof is an adaption of
  \citet[Theorem~2]{Favaro2012}. We just have to take care that here, we proved
  $Q_p \rightarrow Q$ vaguely in \cref{thm:2}, which does not necessarily imply
  that $f^{(Q_p)}(x) \rightarrow f^{(Q)}(x)$ pointwise. But by assumption,
  $x\mapsto K(x;y)$ is continuous and vanishes outside a compact set, and it is
  easily seen that the sequence of total mass $|Q|(\cdot;\omega)$ is
  almost-surely bounded, then by \citet[Theorem~30.6]{Bauer2001} (which remains
  valid for signed measures), we have $f^{(Q_p)}(x)\rightarrow f^{(Q)}(x)$
  pointwise, almost-surely. The end of the proof is identical to
  \citet[Theorem~2]{Favaro2012} for convergence in $L^1$, and extension to $L^q$
  with $1\leq q < +\infty$ is straightforward.
\end{proof}

\section{Proofs of \cref{sec:general-result}}
\label{sec:proofs-crefs-result}

\begin{proof}[Proof of \cref{thm:1}]
  The proof is similar to \citet[theorem~5]{GhosalVanDerVaart2007}. The
  event $A_n$ that
  $\int \prod_{i=1}^n\frac{dP_{\theta,i}}{dP_{\theta_0,i}}(Y_i)\, d\Pi(\theta)
  \geq \e^{-2n\epsilon_n^2/2}$ satisfies $P_{\theta_0}^n(A_n^c) \rightarrow 0$
  by Lemma~10 in \citet{Ghosal2007} and assumptions on $\Pi$. Therefore,
  \begin{align*}
    P_{\theta_0}^n\Pi(\Theta_n^c|Y_1,\dots,Y_n)
    &\leq P_{\theta_0}^n[\Pi(\Theta_n^c|Y_1,\dots,Y_n)\Ind_{A_n}] +
      P_{\theta_0}^n(A_n^c)\\
    &\leq e^{2n\epsilon_n^2} P_{\theta_0}^n\int_{\Theta_n^c}
      \prod_{i=1}^n\frac{dP_{\theta,i}}{dP_{\theta_0,i}}(Y_i)\,d\Pi(\theta) +
      P_{\theta_0}^n(A_n^c)\\
    &\leq e^{2n\epsilon_n^2}\Pi(\Theta_n^c) + P_{\theta_0}^n(A_n^c)
      \rightarrow 0,
  \end{align*}
  where the last lines follows by Fubini's theorem. For $0 < \alpha_j \leq 1$,
  and $n$ large enough, the \cref{lem:5} states the existence of tests functions
  $\psi_{n,j}$ such that
  \begin{gather*}
    P_{\theta_0}^n\psi_{n,j} \leq
    2\alpha_jN(M\epsilon_n,\Theta_{n,j},\rho_n)\e^{-KM^2 n\epsilon_n^2}, \qquad
    P_{\theta}^n(1- \psi_{n,j}) \leq \alpha_j^{-1}\e^{-KM^2n\epsilon_n^2},
  \end{gather*}
  for all $\theta \in \Theta_{n,j}$ with $\rho_n(\theta,\theta_0) >
  11 M\epsilon_n$. Letting
  $U_{\epsilon} := \Set{\theta \in \Theta \given \rho_n(\theta_0,\theta) >
    11M\epsilon_n}$,
  \begin{multline*}
    P_{\theta_0}^n[\Pi(U_{\epsilon}\cap\Theta_{n,j}|Y_1,\dots,Y_n)\Ind_{A_n}]\\
    \begin{aligned}
      &\leq P_{\theta_0}^n\psi_{n,j} + P_{\theta_0}^n\left((1-\psi_{n,j})
        \int_{U_{\epsilon}\cap \Theta_{n,j}} \textstyle
        \prod_{i=1}^n\frac{dP_{\theta,i}}{dP_{\theta_0,i}}(Y_i)\,d\Pi(\theta)
      \right) \e^{2K_1 n\epsilon_n^2}\\
      &\leq P_{\theta_0}^n\psi_{n,j} + \sup_{U_{\epsilon}\cap \Theta_{n,j}}
      P_{\theta}^n(1 - \psi_{n,j})
      \Pi(\Theta_{n,j}) \e^{2K_1n\epsilon_n^2}\\
      &\leq 2\alpha_jN(M\epsilon_n,\Theta_{n,j},\rho_n) \e^{-KM^2n\epsilon_n^2}
      + \alpha_j^{-1}\Pi(\Theta_{n,j})\e^{-(KM^2-2)n\epsilon_n^2/2},
    \end{aligned}
  \end{multline*}
  where we used Fubini's theorem again. Put
  $\alpha_j = \sqrt{\Pi(\Theta_{n,j})/N(M\epsilon_n,\Theta_{n,j},\rho_n})$
  (notice that $\alpha_j \leq 1$) and sum over $j$ to obtain the result in view
  of the last equation.
\end{proof}

\subsection{Existence of tests}
\label{sec:existence-tests-2}

Here we construct the test functions required in the proof of \cref{thm:1}. We
proceed in two steps. First, we construct tests for testing the hypothesis that
$\theta = \theta_0$ against $\theta$ belongs to a ball of radius $\epsilon/12$
centered at $\theta_1$ with $\rho_n(\theta_0,\theta_1) > \epsilon$ ; then in
\cref{lem:5} we construct the tests used in the proof of \cref{thm:1}.

Let $\theta_0 = (f_0,\sigma_0)$, $\theta_1 = (f_1,\sigma_1)$,
$\theta_{10} = (f_1,\sigma_0)$,
$\delta = \sqrt{\epsilon^2 + (108/n)\log(1/\alpha)}$, and define,
\begin{gather*}
  A_n := \Set*{y \in \reals^n \given \sum_{i=1}^n\log
    \frac{dP_{\theta_{0},i}}{dP_{\theta_{10},i}}(y_i) <
    -\frac{n\epsilon^2}{96\sigma_0^2} + 2 \log \alpha },\\
  B_n^c := \Set*{y \in \reals^n \given n(1-\delta/3) \leq
    \sum_{i=1}^n\left(\frac{ y_i - f_0(x_i)}{\sigma_0}\right)^2 \leq n(1+
    \delta/3)}.
\end{gather*}
Then we construct the sequence $(\phi_n)_{n\geq 0}$ as
\begin{multline*}
  \phi_n(Y_1,\dots,Y_n)\\
  := \Ind_{A_n}(Y_1,\dots,Y_n) + \Ind_{B_n}(Y_1,\dots,Y_n) -
  \Ind_{A_n}(Y_1,\dots,Y_n) \Ind_{B_n}(Y_1,\dots,Y_n).
\end{multline*}

\begin{proposition}
  \label{pro:10}
  Let $K = 3(32 \vee 4\sigma_0^2)^{-1}$. The tests $\phi_n$ defined above
  satisfy $P_{\theta_0}^n\phi_n \leq \e^{-Kn\epsilon^2 / 144}$ and
  $\sup_{\theta\in\Theta\, : \, \rho_n(\theta,\theta_1) <
    \epsilon/12}P_{\theta}^n(1 - \phi_n) \leq \e^{-K n\epsilon^2 / 144}$ for all
  $\theta_1 \in \Theta$ with $\rho_n(\theta_0,\theta_1) > \epsilon$ and all
  $0 < \epsilon \leq 1$.
\end{proposition}
\begin{proof}
  \par \textit{Type I error of $\phi_n$.} It is clear that
  $P_{\theta_0}^n\phi_n \leq P_{\theta_0}^n(A_n) + P_{\theta_0}^n(B_n)$.
  Moreover, by proposition~4 in \citet{Birge2006}, we have
  $P_{\theta_0}^n(A_n) \leq \alpha \e^{-n \epsilon^2/(192 \sigma_0^2)}$, and
  regarding the proof of lemma~7 in \citet{Choi2007}, the bound
  $P_{\theta_0}^n(B_n) \leq 2\e^{-n \delta^2/108} = 2\alpha
  \e^{-n\epsilon^2/108}$ holds for $n$ sufficiently large.

  \par \textit{Type II error of $\phi_n$.} Let $\theta = (f,\sigma)$ be such
  that $\rho_n(\theta,\theta_1) \leq \epsilon / 12$.  Clearly,
  $P_{\theta}^n(1 - \phi_n) = P_{\theta}^n(1 - \Ind_{A_n})(1 - \Ind_{B_n}) \leq
  P_{\theta}^n(A_n^c) \wedge P_{\theta}^n(B_n^c)$.
  We should consider two situations, either
  $|\log \sigma_0 - \log \sigma_1| \leq \epsilon / 2$, or
  $|\log \sigma_0 - \log \sigma_1| > \epsilon /2$.
  \begin{itemize}
    \item If $|\log \sigma_0 - \log \sigma_1| \leq \epsilon / 2$, then
    $\rho_n(\theta_0,\theta_1) > \epsilon$ implies
    $\|f_0 - f_1\|_{2,n} > \epsilon / 2$, and for all $\theta$ with
    $\rho_n(\theta,\theta_1) \leq \epsilon / 12$, it is clear that
    $\|f - f_1\|_{2,n} \leq \epsilon / 12$. It follows from proposition~4 in
    \citet{Birge2006} that
    \begin{align*}
      P_{\theta}^n(A_n^c)
      &\leq
        \exp\left[-\frac{n\|f_0 - f_1\|_{2,n}^2 - n\epsilon^2/8 +
        24\sigma_0^2\log \alpha}{24\sigma_0^2}
        \right]
        \leq
        \frac{1}{\alpha}
        \exp\left[-\frac{n\epsilon^2}{64\sigma_0^2} \right].
    \end{align*}
    \item If $|\log \sigma_0 - \log \sigma_1| > \epsilon / 2$, then
    $\rho_n(\theta,\theta_1) \leq \epsilon / 12$ implies
    $|\log \sigma - \log \sigma_0| > 5 \epsilon / 12$. We should again subdivise
    this case, considering either $\sigma/\sigma_0 \geq 1$ or not. For both
    cases we mimick and adapt the proof of lemma~7 in \citet{Choi2007}.
    \begin{itemize}
      \item If $\sigma/\sigma_0 \geq 1$, because
      $|\log \sigma - \log \sigma_0| > \epsilon / 3$ we have
      $\sigma > \sigma_0 \e^{\epsilon / 3}$, and thus
      $\sigma > (1 + \epsilon/3)\sigma_0$ for any $\epsilon > 0$. Let $W \sim
      \chi_n^2$ and let $W'$ have a noncentral $\chi^2$ distribution with $n$
      degrees of freedom and noncentrality parameter $\sum_{i=1}^n(f(x_i) -
      f_0(x_i))^2$. Then,
      \begin{align*}
        P_{\theta}^n(B_n^c)
        &\leq P_{\theta}^n\left(\sum_{i=1}^n\left(\frac{Y_i -
          f_0(x_i)}{\sigma_0}\right)^2 \leq n\left(1 + \frac{\delta}{3}\right)
          \right)\\
        &=\Pr\left( W' \leq n \frac{\sigma_0^2}{\sigma^2}\left(1 +
          \frac{\delta}{3} \right) \right)
          \leq \Pr\left( W \leq n \frac{\sigma_0^2}{\sigma^2}\left(1 +
          \frac{\delta}{3} \right) \right).
      \end{align*}
      But whenever $0 < \alpha \leq 2$, we have
      \begin{equation*}
        \frac{\sigma_0^2}{\sigma^2}\left(1 + \frac{\delta}{3}\right)
        \leq \frac{1 + \delta/3}{(1+\epsilon/3)^2}
        \leq \frac{1}{1 + \epsilon/3} + \frac{(108/n)\log(1/\alpha)}{6\epsilon(1
          + \epsilon/3)^2}.
      \end{equation*}
      Therefore, by Markov's inequality we get for all $t < 1/2$
      \begin{align*}
        P_{\theta}^n(B_n^c)
        &\leq
          \exp\left\{-t\frac{108
          \log(1/\alpha)}{6\epsilon(1+\epsilon/3)^2}\right\}
          \exp\left\{-\frac{nt}{1 + \epsilon/3}\right\}
          (1 - 2t)^{-n/2}.
      \end{align*}
      Choosing $t = -\epsilon/18$ leads to
      \begin{align*}
        P_{\theta}^n(B_n^c)
        &\leq \exp\left\{\frac{\log(1/\alpha)}{(1+\epsilon/3)^2}\right\}
          \exp\left\{\frac{n}{2}\left(\frac{\epsilon/9}{1+\epsilon/3} -
          \log(1+\epsilon/9) \right) \right\}\\
        &\leq \frac{1}{\alpha} \exp\left\{- \frac{7 n\epsilon^2}{648}\right\}
          \leq \frac{1}{\alpha} \exp\left\{- \frac{n\epsilon^2}{93}\right\},
      \end{align*}
      because we have $0 < \epsilon \leq 1$. This concludes the proof when
      $\sigma / \sigma_0 \geq 1$.
      \item On the other direction, $\sigma/\sigma_0 < 1$ and
      $|\log \sigma - \log \sigma_0| > 5\epsilon / 12$ imply that
      $\sigma < (1- \epsilon/3)\sigma_0$ for any $0 < \epsilon \leq 1$.  Using
      the same strategy as in the previous item it is possible to show that the
      bound $P_{\theta}^n(B_n^c) \leq
      (1/\alpha)\e^{-n\epsilon^2/1536}$ holds. \qedhere
    \end{itemize}
  \end{itemize}
\end{proof}

\begin{lemma}
  \label{lem:5}
  Let $\Theta_n \subset \Theta$ and $K := 3(32 \vee 4\sigma_0^2)^{-1}$. Then for
  any $0<\alpha \leq 1$ there exists a collection of tests functions
  $(\psi_n)_{n\geq 1}$ such that for any $0 < \epsilon \leq 1/12$ and any
  $n\geq 1$
  \begin{gather*}
    P_{\theta_0}^n\psi_n \leq 2\alpha N(\epsilon,\Theta_n,\rho_n) \e^{-K
      n\epsilon^2}, \qquad \sup_{\theta\in \Theta_n\,:\,\rho_n(\theta,\theta_0)
      > 11 \epsilon }P_{\theta}^n(1 - \psi_n) \leq \alpha^{-1} e^{- K n
      \epsilon^2}.
  \end{gather*}
\end{lemma}
\begin{proof}
  Let $N \equiv N(\epsilon/12,\Theta_n,\rho_n)$ denote the number of balls of
  radius $\epsilon/12$ needed to cover $\Theta_n$. Let $(B_1,\dots,B_N)$ denote
  the corresponding covering and $(\zeta_1,\dots,\zeta_N)$ denote the centers of
  $(B_1,\dots,B_N)$. Now let $J$ be the index set of balls $B_j$ with
  $\rho_n(\theta_0,\zeta_j) > \epsilon$. Using \cref{pro:10} for
  $0<\epsilon\leq 1$ and for any ball $B_j$ with $j\in J$, we can build a test
  function $\phi_{n,j}$ satisfying
  \begin{gather*}
    P_{\theta_0}^n\phi_{n,j} \leq 2\alpha\e^{-Kn\epsilon^2 / 144}, \qquad
    \sup_{\theta \in B_j} P_{\theta}^n(1 - \phi_{n,j}) \leq \alpha^{-1}\e^{-K
      n \epsilon^2 / 144},
  \end{gather*}
  Let $\psi_n := \max_{j\in J} \phi_{n,j}$. Then
  $P_{\theta_0}^n\psi_n \leq \sum_{j\in J}P_{\theta_0}^n\phi_{n,j} \leq 2\alpha
  N(\epsilon/12,\Theta_n, \rho_n) \e^{- K n \epsilon^2 / 144}$ and also
  $P_{\theta}^n(1 - \psi_n) \leq \min_{j\in J}\sup_{\theta'\in
    B_j}P_{\theta'}^n(1 - \phi_{n,j}) \leq \alpha^{-1}\e^{-K n \epsilon^2 /
    144}$ for any $\theta \in \Theta_n$ with
  $\rho_n(\theta,\theta_0) > 11\epsilon/12$.
\end{proof}

\section{Proof of \cref{thm:3}}
\label{sec:proof-rates-loc-scale}

We prove \cref{thm:3} by verifying the set of sufficient conditions established
in \cref{thm:1}.

\subsection{Sieve construction}
\label{sec:sieve-constr-locat}

For constants $H,M > 0$ to be determined later, we define the sets
\begin{gather*}
  \mathcal{D}_n := \Set*{A \in \mathcal{E}_s \given n^{-1/a_2} \leq \lambda_i(A)
    \leq n^{-1/a_2}(1 + M\epsilon_n/n)^{n^2}, \quad i=1,\dots,d },\\
    \Theta_n
  := \Set*{ (f,\sigma) \given
  \begin{array}{l}
    n^{-2/a_8} < \sigma^{2} \leq n^{-2/a_8}(1+M\epsilon_n)^n,\ f(x) = \int K_A(x
    - \mu)Q(dAd\mu),\\
    Q = \sum_{i=1}^{\infty}u_i\delta_{A_i,\mu_i},\ \supp Q = \mathcal{E}_{s}
    \times [-2S,2S]^d,\ \sum_{i=1}^{\infty}|u_i|\leq n,\\
    \#\Set{i \given |u_i| > n^{-1},\ A_i \in \mathcal{D}_n} \leq
    Hn\epsilon_n^2 / \log n,\\
    \sum_{i=1}^{\infty}|u_i|\Ind\Set{A_i\notin \mathcal{D}_n} \leq
    M\epsilon_n,\quad \sum_{i=1}^{\infty}|u_i|\Ind\Set{|u_i| \leq n^{-1}} \leq
    M\epsilon_n
  \end{array}
}.
\end{gather*}

In the sequel, we assume without loss of generality that the jumps of $Q$ in the
definition of $\Theta_n$ are ordered so that there is no jump with
$|u_i| > n^{-1}$ and $A_i\in \mathcal{D}_n$ when $i>H n\epsilon_n^2 / \log
n$. Moreover, we consider the following partition of $\Theta_n$. Let $H_n$ the
largest integer smaller than $Hn\epsilon_n^2 / \log n$. Then for any
$j = (j_1,\dots,j_{H_n})\in \Set{1,2,\dots}^{H_n}$, inspired by
\citet[theorem~2]{Canale2013} we define the slices
\begin{equation*}
  \Theta_{n,j} :=
  \Set*{(f,\sigma) \in \Theta_n \given
    n^{2^{j_i-1}} < \lambda_1(A_i)/\lambda_{d}(A_i) \leq n^{2^{j_i}}\quad
    \forall i \leq H_n }.
\end{equation*}

\begin{lemma}
  \label{lem:9b}
  Assume that there is $0 < \gamma_1 < 1$ such that
  $\epsilon_n^2 \geq n^{-\gamma_1}$ for all $n$ large enough. Then for
  $H = 6(1 - \gamma_1)^{-1}$ it holds
  $\Pi(\Theta_n^c) \lesssim \exp(-3n\epsilon_n^2)$ as $n\rightarrow \infty$.
\end{lemma}
\begin{proof}
  From the definition of $\Theta_n$, it is clear that
  \begin{multline}
    \label{eq:7}
    \Pi(\sievec) \leq \Pi\left( \#\Set{i \given |u_i| > n^{-1}} > H
      n\epsilon_n^2/\log n \right)
    +\Pi\left( \textstyle\sum_{i=1}^{\infty}|u_i| > n \right)\\
    +\Pi\left( \textstyle\sum_{i=1}^{\infty}|u_i| \Ind\Set{|u_i| \leq
        n^{-1}} > M\epsilon_n \right) +\Pi\left(
      \textstyle\sum_{i=1}^{\infty}|u_i| \Ind\Set{(A_i,\mu_i)
        \notin \mathcal{D}_n} > M\epsilon_n \right)\\
    + P^{\sigma}( \sigma^2 \leq n^{-2/a_8}) + P^{\sigma}(\sigma^2 >
    n^{-2/a_8}(1+ M\epsilon_n)^n).
  \end{multline}
  The bounds on the two last terms are obvious in view of \cref{eq:29,eq:27}.

  \par By the superposition theorem \citep[section~2]{KINGMAN1992}, for any
  measurable set $A\subseteq \mathcal{E}\times \reals^d$ we have
  $Q(A) := Q_1(A) + Q_2(A)$ where $Q_1$ and $Q_2$ are independent signed random
  measures with total variation having Laplace transforms (for all measurable
  $A\subseteq \mathcal{E}\times \reals^d$ and all $t\in\reals$ for which the
  integrals in the expression converge)
  \begin{gather}
    \label{eq:1}
    \Esp \e^{t|Q_1|(A)} = \exp\left\{ 2\alpha F(A)\int_{n^{-1}}^{\infty}(\e^{tx}
      -1)x^{-1}\e^{-\eta x}dx\right\},\\
    \label{eq:22}
    \Esp \e^{t |Q_2|(A)} = \exp\left\{2\alpha F(A) \int_0^{n^{-1}}(\e^{tx} -
      1)x^{-1}\e^{-\eta x} dx \right\}.
  \end{gather}
  The random measures $Q_1$ and $Q_2$ are almost-surely purely atomic, the
  magnitudes of the jumps of $Q_1$ are all $\geq n^{-1}$, whereas $Q_2$ has
  jumps magnitudes all $< n^{-1}$ (almost-surely). Also, the number of jumps of
  $Q_1$ is distributed according to a Poisson law with intensity
  $\alpha E_1(n^{-1} / \eta)$, where $E_1$ is the exponential integral $E_1$
  function. Recalling that $E_1(x) \asymp \gamma + \log(1/x)$ for $x$ small, it
  follows
  $\alpha(\gamma + \log \eta) \leq \alpha E_1(n^{-1}/\eta) \leq 2\alpha \log n
  \ll H n\epsilon_n^2 / \log n$ when $n$ is large. Then using Chernoff's bound
  on Poisson law, we get
  \begin{multline*}
    \Pi\left(\#\Set{i \given |u_i| > n^{-1}} > Hn\epsilon_n^2/\log n\right)\\
    \begin{aligned}
      &\leq \e^{-\alpha E_1(n^{-1}/\eta)}\frac{(\e \alpha
        E_1(n^{-1}/\eta)^{Hn\epsilon_n^2/\log n}}{(Hn\epsilon_n^2/\log n)^{H
          n\epsilon_n^2 / \log n}}\\
      &\leq (\eta e^{\gamma})^{\alpha}\exp\left\{-\frac{Hn\epsilon_n^2}{\log n}
        \left( \log \frac{Hn\epsilon_n^2}{\log n} - \log(2e\alpha\log n)\right)
      \right\}.
    \end{aligned}
  \end{multline*}
  But,
  \begin{equation*}
    \log \frac{Hn\epsilon_n^2}{\log n} - \log(2e\alpha\log n)
    \geq (1 - \gamma_1)\log n - 2\log \log n + \log\frac{H}{2e\alpha},
  \end{equation*}
  which is in turn greater than $(1/2)(1-\gamma_1)\log n$ when $n$ becomes
  large. This gives the proof for the first term of the rhs of \cref{eq:7}.

  \par Regarding the second term of the rhs of \cref{eq:7}, it suffices to
  remark that the random variable $\sum_{i=1}^n|u_i|$ has Gamma distribution
  with parameters $(2\alpha,\eta)$. Then the upper bound on
  $\Pi(\sum_{i=1}^n|u_i| > n)$ follows from Markov's inequality. With the
  same argument, we have that the random variable
  $\sum_{i=1}^{\infty}|u_i|\Ind\Set{|u_i| \leq n^{-1}}$ is equal in distribution
  to $|Q_2|(\mathcal{E}\times\mathbb \reals^d)$, thus the bound for the fourth
  term of the rhs of \cref{eq:7} follows from Markov's inequality and
  \cref{eq:22}, because
  \begin{align*}
    \Pi\left( e^{3n\epsilon_n|Q_2|} > e^{3n\epsilon_n^2}\right)
    \leq \e^{-3n\epsilon_n^2}\exp\left\{2\alpha
    \int_{0}^{n^{-1}}(e^{n\epsilon_n x} - 1) x^{-1}\e^{-\eta x }dx \right\}
    \lesssim e^{-3n\epsilon_n^2}.
  \end{align*}

  \par The fifth term of the rhs of \cref{eq:7} is bounded using Chebychev's
  inequality. Indeed, with the same argument as before, the random variable
  $X:=\sum_{i=1}^n|u_i| \Ind\Set{A_i \notin \mathcal{D}_n}$ has Gamma
  distribution with parameters $(2\alpha F_A(\mathcal{D}_n^c),\eta)$. Hence for
  $n$ sufficiently large we have
  $\mathbb E X = 2\alpha F_A(\mathcal{D}_n^c)/\eta \leq \epsilon_n /2$, and
  \begin{align*}
    \Pi(X >\epsilon_n)
    \leq \Pi(X - \mathbb E X > \epsilon_n/2)
    \leq \frac{8\alpha F_A(\mathcal{D}_n^c)}{\eta^2\epsilon_n^2}.
  \end{align*}
  Then the result follows from \cref{eq:19,eq:18}.
\end{proof}

\begin{lemma}
  \label{lem:6}
  Let $\epsilon_n \rightarrow 0$ with $n\epsilon_n^2 \rightarrow \infty$ and
  $K = 3(32 \vee 4\sigma_0^2)$. Then there exists $M > 0$ such that it holds
  $\sum_{j}\sqrt{N(M\epsilon_n,\Theta_{n,j},\rho_n)} \sqrt{\Pi(\Theta_{n,j})}
  \e^{-(K M^2 - 2)n\epsilon_n^2} \rightarrow 0$.
\end{lemma}
\begin{proof}
  Define the random measures $Q_1$ and $Q_2$ as in the proof of
  \cref{lem:9b}. Then using the Poisson construction of $Q_1$ (see for instance
  \citet[section~2.3.1]{Wolpert2011}), it follows from \cref{eq:31} that for any
  $j \in \Set{1,2,\dots}^{H_n}$
  \begin{align*}
    \Pi(\Theta_{n,j})
    &\leq \textstyle \prod_{i\leq H_n} F_A(A\ :\ \lambda_1(A)/\lambda_d(A)\geq
      n^{2^{j_i-1}})
      \leq b_6^{H_n}\prod_{i\leq H_n} n^{-\kappa^{*}2^{j_i-1}}.
  \end{align*}
  Moreover, using \cref{pro:4b} we can find a constant $C> 0$ independent of $M$
  such that
  $N(M\epsilon_n,\Theta_{n,j},\rho_n) \leq \e^{-2CHn\epsilon_n^2}
  n^{d(d-1)/2\sum_{i\leq H_n}2^{j_i}}$ when $n$ is large. Therefore,
  \begin{equation*}
    \textstyle
    \sqrt{N(M\epsilon_n,\Theta_{n,j},\rho_n)}\sqrt{\Pi(\Theta_{n,j})}
    \leq
    \exp\left\{Hn\epsilon_n^2\left(C + \frac{\log b_6}{2\log n}\right) \right\}
    \prod_{i\leq H_n}
    n^{\frac 12[d(d-1) - \kappa^{*}]2^{j_i-1}}.
  \end{equation*}
  For $n$ large enough we have $\log b_6 \leq 2C\log n$ ; then provided
  $\kappa^{*} > d(d-1)$, we can sum over $j\in \Set{1,2,\dots}^{H_n}$ the last
  expression to get
  \begin{align*}
    \textstyle
    \sum_{j}\sqrt{N(M\epsilon_n,\Theta_{n,j},\rho_n)}\sqrt{\Pi(\Theta_{n,j})}
    &\leq
    \exp\left\{2C Hn\epsilon_n^2\right\}
    \left( \textstyle \sum_{k\geq 1} n^{\frac 12[d(d-1) - \kappa^{*}]2^{k-1}}
      \right)^{H_n}\\
    &\leq \exp\left\{H(2C + \kappa^{*}/2)n\epsilon_n^2 \right\}.
  \end{align*}
  Now choose $M > 0$ satisfying $KM^2 > 2 + H(2C + \kappa^{*}/2)$ to obtain
  the conclusion of the lemma.
\end{proof}

\begin{proposition}
  \label{pro:4b}
  For $n$ large enough there is a constant $C > 0$ independent of $M$ such that
  for any sequence $\epsilon_n \rightarrow 0$ with
  $n\epsilon_n^2 \rightarrow \infty$, the following holds for any
  $j \in \Set{1,2,\dots}^{H_n}$.
  \begin{equation*}
    \log N(M\epsilon_n,\Theta_{n,j},\rho_n)
    \leq CH n \epsilon_n^2 + \frac{d(d-1)}{2}\log n \sum_{i\leq H_n}2^{j_i}.
  \end{equation*}
\end{proposition}
\begin{proof}
  The proof is based on arguments from \citet{Shen2013}, it uses the fact that
  the covering number $N(M\epsilon_n,\sieve,\rho_n)$ is the minimal cardinality
  of an $M\epsilon_n$-net over $\sieve$ in the distance $\rho_n$. Let
  $\delta_n := M\epsilon_nn^{-(1+1/a_2)}$, $\widehat{R}_n$ be a $\delta_{n}$-net
  of $[-2S,2S]^d$, $\widehat{\Delta}_n$ be a $M\epsilon_n$-net of
  $\Set{(u_1,\dots,u_{H_n}) \in \reals^{H_n}\given \sum_{i=1}^{H_n}|u_i| \leq
    n}$ in the $\ell_1$-distance, and
  $\widehat{S}_n := \Set{\sigma > 0 \given \sigma^2 = n^{-2/a_8}(1 +
    M\epsilon_n)^k,\ k\in \mathbb N,\ k\leq n}$. Also, for any $k\geq 1$ let
  $\widehat{\mathcal{O}}_k$ be a $n^{-(2^{k}+1)}M\epsilon_n$-net of the group of
  $d\times d$ orthogonal matrices equipped with spectral norm $\|\cdot\|$, and
  define
  \begin{gather*}
    \widehat{\mathcal{D}}_{n,k} := \Set*{A\in \mathcal{D}_n \given
      \begin{array}{l}
        A = P \Lambda P^{\top},\ P\in
        \widehat{\mathcal{O}}_k,\ \Lambda =
        \mathrm{diag}(\lambda_1,\dots,\lambda_d),\\
        \lambda_j = n^{-1/a_2}(1 + M\epsilon_n/n)^k,\ k\in \mathbb
        N,\ k\leq n^2,\ j=1,\dots,d
      \end{array}
    }.
  \end{gather*}
  Pick $(f,\sigma) \in \Theta_{n,j}$ with
  $f(x) = \sum_{i=1}^{\infty}u_i\, K_{A_i}(x - \mu_i)$. Clearly we can find
  $\widehat{u} \in \widehat{\Delta}$ such that
  $\sum_{i\leq H_n}|u_i - \widehat{u}_i| \leq M\epsilon_n$,
  $\widehat{\mu} \in \widehat{R}_n^{H_n}$ such that
  $|\mu_i - \widehat{\mu}_i|_d \leq \delta_n$ for all $i=1,\dots, H_n$, and
  $\widehat{\sigma} \in \widehat{S}_n$ such that
  $|\log \sigma - \log \widehat{\sigma}| \leq M\epsilon_n$. We also claim that
  we can find $\widehat{A}_{i} \in \widehat{\mathcal{D}}_{n,j_i}$ such that
  $\|I - A^{-1}_i\widehat{A}_i^{-1}\| \leq 3dM\epsilon_n/n$ for all $i\leq
  H_n$. We defer the proof of the claim to later. Let
  $\widehat{f}(x) = \sum_{i\leq H_n}\widehat{u}_i\, K_{\widehat{A}_i}(x -
  \widehat{\mu}_{i})$ denote the function built from the parameters chosen as
  above ; it follows
  \begin{align*}
    \|f - \widehat{f}\|_{2,n}
    &\leq
      \sum_{i> H_n}|u_i|
      + \sum_{i\leq H_n}|u_i - \widehat{u}_i|
      +\sum_{i\leq H_n}|u_i| \|K_{A_i}(\cdot -\mu_i) -
      K_{\widehat{A}_i}(\cdot - \widehat{\mu}_i)\|_{2,n}\\
    &\leq
      2M\epsilon_n
      + C'\sum_{i\leq H_n}|u_i| \|I - A_i^{-1} \widehat{A}_i\|
      + C'\sum_{i\leq H_n} |u_i| \|A^{-1}_i\| |\mu_{i} - \widehat{\mu}_i|_d\\
    &\leq M(2 + C' + 3C'd)\epsilon_n,
  \end{align*}
  where the two last inequalities hold by \cref{pro:11} for a constant $C'> 0$
  depending only on $g$, and because $\|A_i^{-1}\| \leq n^{1/a_2}$ for all
  $i\leq H_n$. Thus a $(2+C'+3C'd)M\epsilon_n$-net of $\Theta_{n,j}$ in the
  distance $\rho_n$ can be constructed with $(\widehat{f},\widehat{\sigma})$ as
  above. Recall that $\# \widehat{R}_n \leq (4S/\delta_n)^d$,
  $\# \widehat{\Delta}_n \leq (n/(M\epsilon_n))^{H_n}$, $\# \widehat{S}_n = n$
  and $\# \widehat{\mathcal{O}}_k \leq
  (n^{-(2^k+1)}M\epsilon_n)^{-d(d-1)/2}$. It turns out that
  $\# \widehat{\mathcal{D}}_{n,k} \leq n^{2d}\times \#
  \widehat{\mathcal{O}}_k$. Then the total number of
  $(\widehat{f},\widehat{\sigma})$ is bounded by a multiple constant of
  \begin{align*}
    n\times \left(\frac{4S}{\delta_n}\right)^{H_n} \times
    \left(\frac{n}{M\epsilon_n}\right)^{H_n} \prod_{i\leq H_n} \left( n^{2d}
    \times \left( \frac{n^{2^{j_i}+1}}{M\epsilon_n}\right)^{d(d-1)/2}\right).
  \end{align*}
  Finally, $H_n |\log M| \ll H_n \log n$ when $n$ is large proving that the
  constant $C>0$ can be chosen independent of $M$, and the constant factor
  $2+C'+3C'd$ can be absorbed into the bound.

  \par It remains to prove that for any $A\in \mathcal{D}_{n}$ with
  $\lambda_1(A) / \lambda_d (A) \leq n^{2^k}$ we can find
  $\widehat{A} \in \widehat{\mathcal{D}}_{n,k}$ such that
  $\|I - A^{-1}\widehat{A}\| \leq 3d M\epsilon_n / n$. Let
  $A =: P \Lambda P^{\top}$ denote the spectral decomposition of $A$ (recall
  that $A$ is symmetric). Clearly, we can find a matrix
  $\widehat{A} := \widehat{P}\widehat{\Lambda}\widehat{P}^{\top}$ in
  $\widehat{\mathcal{D}}_{n,k}$ with
  $\|P - \widehat{P}\| \leq n^{-(2^k+1)}M\epsilon_n$ and
  $1 \leq \lambda_j(\Lambda)/\lambda_j(\widehat{\Lambda}) \leq 1 +
  M\epsilon_n/n$ for all $j=1,\dots,d$. Let
  $\widetilde{A} := \widehat{P} \Lambda \widehat{P}^{\top}$ and remark that
  \begin{align}
    \notag
    \|I - A^{-1}\widehat{A}\|
    &\leq \|I - A^{-1}\widetilde{A}\| + \|A^{-1}\widetilde{A}\|\|I -
      \widetilde{A}^{-1}\widehat{A}\|\\
    \label{eq:30}
    &\leq \|I - A^{-1}\widetilde{A}\| + \|I -
      \widetilde{A}^{-1}\widehat{A}\|(1 + \|I - A^{-1}\widetilde{A}\|).
  \end{align}
  Let $B := P^{\top}\widehat{P} - I$, so that
  $\|B\|_{\max} \leq \|B\| \leq \|P^{\top}\widehat{P} - I\| \leq \|P -
  \widehat{P}\| \leq n^{-(2^k+1)}M\epsilon_n$, and
  $I - A^{-1}\widetilde{A} = P(B - \Lambda^{-1}B\Lambda)\widehat{P}^{\top}$. It
  follows,
  \begin{align*}
    \|I - A^{-1}\widetilde{A}\|
    &\leq \|B - \Lambda^{-1}B\Lambda\|
      \leq d\|B\|_{\max} \frac{\lambda_1(\Lambda)}{\lambda_d(\Lambda)}
      \leq dM\epsilon_n/n,
  \end{align*}
  because the entries of $B - \Lambda^{-1}B\Lambda$ are equal to
  $B_{ij}(1 - \Lambda_j/\Lambda_i)$ and $\|\cdot\| \leq
  d\|\cdot\|_{\max}$. Moreover,
  $I - \widetilde{A}^{-1}\widehat{A} = \widehat{P}(I -
  \Lambda^{-1}\widehat{\Lambda})\widehat{P}^{\top}$ implies
  $\| I - \Lambda^{-1}\widehat{\Lambda}\| \leq dM\epsilon_n/n$. Then the
  conclusion follows from \cref{eq:30}.
\end{proof}

\subsection{Approximation of functions}
\label{sec:appr-smooth-funct}

In order to prove the prior positivity of Kullback-Leibler balls around
$\theta_0$, we need to approximate $f_0 \in \mathcal{C}^{\beta}[-S,S]^d$ by
finite location-scale mixtures of kernels. We mostly follow the approach of
\citet[lemma~3.4]{DeJongeVanZanten2010}.

Nevertheless, as mentioned in \citet{DeJongeVanZanten2010}, we shall extend
$f_0$ defined on $[-S,S]^d$ onto a (smooth) function defined on $\mathbb R^d$ to
be able to approximate properly $f_0$; otherwise we could have troubles at the
boundaries of $[-S,S]^d$. Clearly, without any precaution,
$h^{-d}K_{hI}*f_0(x) \rightarrow f_0(x)/2$ as $h\rightarrow 0$ when $x$ belongs
to the boundary of $[-S,S]^d$. \citet{DeJongeVanZanten2010} assume that the
covariates are spread onto $[a,b]^d$ with $a > -S$ and $b < S$ and extend $f_0$
by multiplying it by a smooth function that equal $1$ on $[a,b]^d$ and $0$
outside $[-S,S]^d$. Here we assume that the covariates are spread onto
$[-S,S]^d$ and we use Whitney's extension theorem \citep{Whitney1934} to find a
function $\widetilde{f}_0 : \mathbb R^d \rightarrow \mathbb R$ such that
$\widetilde{f}_0 \in \mathcal{C}^{\beta}(\mathbb R^d)$ and
$D^{\alpha}\widetilde{f}_0(x) = D^{\alpha}f_0(x)$ for all $x\in [-S,S]^d$ and
all $|\alpha| \leq \beta$. Then we apply the method of
\citet[lemma~3.4]{DeJongeVanZanten2010} to $\widetilde{f}_0$. We find this
approach more elegant since we do not have to assume that $f_0$ is defined on a
larger set than the support of the covariates.

For each $\alpha \in \mathbb N^d$, let
$\mathfrak{m}_{\alpha}^h := h^{-d}\int x^{\alpha}K_{hI}(x)\,dx$. For
$\alpha \in \mathbb N^d$ with $|\alpha| \geq 1$, define two sequences of numbers
by the following recursion. If $|\alpha| = 1$ set $c_{\alpha} = 0$ and
$d_{\alpha} = - 1 / \alpha!$, and for $|\alpha| \geq 2$
define
\begin{gather}
  \label{eq:25}
  c_{\alpha} := \sum_{\substack{l + k = \alpha\\ |l|\geq 1,\, |k|\geq
    1}}\frac{(-1)^{|\alpha|}}{\alpha!}\left(\frac{\mathfrak{m}_{l}^h
      \mathfrak{m}_k^{h}}{\mathfrak{m}_{\alpha}^h}\right)d_k,\quad
  d_{\alpha} := \frac{(-1)^{|\alpha|}}{\alpha!} - c_{\alpha}.
\end{gather}
Given $\beta > 0$, $h > 0$ and $p$ the largest integer strictly smaller than
$\beta$, define
\begin{equation*}
  f_{\beta} := \widetilde{f}_0 - \sum_{1 \leq |\alpha| \leq p}
  d_{\alpha}\mathfrak{m}_{\alpha}^h D^{\alpha}\widetilde{f}_0.
\end{equation*}
\begin{proposition}
  \label{pro:6}
  Let $h > 0$. For any $\beta > 0$ and any function
  $f_0 \in \mathcal{C}^{\beta}[-S,S]^d$ there is a positive constant $M_{\beta}$
  such that $|h^{-d}K_{hI}*f_{\beta}(x) - f_0(x)| \leq M_{\beta}h^{\beta}$ for
  all $x\in \mathbb [-S,S]^d$.
\end{proposition}
\begin{proof}
  Noticing that $\mathfrak{m}_{\alpha}^h \lesssim h^{|\alpha|}$, the proof
  follows from the same argument as in \cite[lemma~2]{Shen2013}, because
  $\widetilde{f}_0(x) = f_0(x)$ for all $x\in [-S,S]^d$.
\end{proof}
The \cref{pro:6} shows that any sufficiently regular function can be
approximated by continuous location mixtures of $K_{hI}$, provided $h$ is chosen
small enough and $g$ has enough finite moments. In the sequel, we will need
slightly more, that is approximating any $\beta$-Hölder continuous function by
\textit{discrete} mixtures of $K_{hI}$~; this is done by discretizing the
convolution operator in the next proposition. Compared to
\citet[lemma~3.1]{GhosalVanDerVaart2001}, we need to take extra cares regarding
the fact that $f_0$ can take negative values, and also to control the ``total
mass'' of the mixing measure.

\begin{proposition}
  \label{pro:13}
  Let $h > 0$ be small enough and $\zeta = 1 \vee 2/(\tau - \gamma \tau)$. There
  exists a discrete mixture $f(x) = \sum_{i=1}^N\alpha_i\, K_{hI}(x - \mu_i)$
  with $N \lesssim h^{-d} (\log h^{-1})^{d(\zeta - 1)}$, $\mu_i \in [-2S,2S]^d$
  for all $i=1,\dots,N$; such that $|f(x) - f_0(x)| \lesssim h^{\beta}$ for all
  $x\in [-S,S]^d$. Moreover $\sum_{i=1}^N |\alpha_i| \lesssim h^{-d}$, and
  $|\mu_i - \mu_j|_d \geq h^{\beta +1}$ for any $i\ne j$.
\end{proposition}
\begin{proof}
  Let $Q$ be the signed measure defined by $A\mapsto \int_Af_{\beta}(y)dy$ for
  any measurable set $A\subseteq \mathbb R^d$. Let
  $M_h := (C_0^{-1}(\beta+d) \log h^{-1})^{1/\tau}$. To any $j \in \mathbb Z^d$
  we associate the cube $B_j := hM_h(j + [0,1]^d)$ and the signed measure $Q_j$
  such that $Q_j(A) := Q(A \cap B_j)$ for all measurable
  $A \subseteq \mathbb R^d$. Let $Q_j^+$, $Q_j^-$ denote respectively the
  positive and negative part of the Jordan decomposition of $Q_j$. It is a
  classical result from \citet{Tchakaloff1957} that we can construct discrete
  (positive) measures $P_{j,k}^+$, $P_{j,k}^-$ each having at most
  $(k+d)!/(k!d!)$ atoms and satisfying
  $\int R(x)Q_j^{\pm}(dx) = \int R(x)P_j^{\pm}(dx)$ for any polynomial $R(x)$ of
  degree $|\alpha| \leq k$. Let
  $\Lambda_h := \Set{j \in \mathbb Z^d \given |j| \leq 1 + S/(hM_h)}$ and for
  any $x \in \mathbb R^d$ let
  $N_x := \Set{j \in \Lambda_h\given \inf\Set{|x - y|_d \given y \in B_j } \leq
    hM_h}$. For the signed measure
  $P_k := \sum_{j\in \Lambda_h}(P_{j,k}^{+} - P_{j,k}^{-})$ the total variation
  of $P_k$ satisfy the bound
  \begin{align*}
    |P_k| \leq \sum_{j\in \Lambda_h}P_{j,k}^{+} + \sum_{j\in
            \Lambda_h}P_{j,k}^{-}
    \leq \sum_{j\in \mathbb Z^{d}}(P_{j,k}^+ + P_{j,k}^-) = |Q|.
  \end{align*}
  Notice that $|Q| < +\infty$ since we have $f_{\beta} \in L^1(\mathbb R^d)$.
  Moreover, letting $P_{j,k} = P_{j,k}^+ - P_{j,k}^-$
  \begin{multline}
    \label{eq:17}
    \int K_{hI}(x-y)(Q-P_k)(dy)
    = \sum_{j\notin \Lambda_h}\int_{B_j}g\left( \frac{x - y}{h} \right)Q_j(dy)\\
    + \sum_{j\in \Lambda_h\backslash N_x}\int_{B_j}g\left( \frac{x - y}{h}
    \right) (Q_j - P_{j,k})(dy) + \sum_{j\in N_x}\int_{B_j}g\left( \frac{x -
        y}{h} \right) (Q_j - P_{j,k})(dy).
  \end{multline}
  By assumptions on $g$, for any $x\in [-S,S]^d$ the first term of the rhs of
  \cref{eq:17} is bounded by $|Q| h^{\beta+d}$. With the same argument, using
  the definition of $N_x$, the second term of the rhs of \cref{eq:17} is bounded
  by $2|Q| h^{\beta+d}$. Regarding the last term, using multivariate Taylor's
  formula we write
  \begin{multline}
    \label{eq:13}
    \int_{B_j}g\left( \frac{x -
        y}{h} \right) (Q_j - P_{j,k})(dy)
    = \sum_{|\alpha| \leq k}\frac{D^{\alpha}g(0)}{\alpha !}
    \int_{B_j}\left(\frac{x - y}{h}\right)^{\alpha}(Q_j - P_{j,k})(dy)\\
    + \int_{B_j}R_k\left(\frac{x-y}{h}\right)(Q_j - P_{j,k})(dy),
  \end{multline}
  where $|R_k(x)| \leq \sup_{|\alpha| = k}|D^{\alpha}g(0)| |x|_d^k / k!$. The
  first term of the rhs of \cref{eq:13} vanishes by construction of
  $P_{j,k}$. For any $j \in N_x$ and any $y\in B_j$ it holds
  $|x - y|_d \leq 2hM_h$~; then using Stirling's formula and assumptions on
  $D^{\alpha}g$ the second term of the rhs of \cref{eq:13} is bounded by
  \begin{equation*}
    \sup_{|\alpha| =k}|D^{\alpha}g(0)| \frac{(2eM_h)^k}{\sqrt{2\pi k}
      k^k}\int_{B_j}|Q_j - P_{j,k}|(dy)
    \leq K_1 \exp\left\{-k (1-\gamma)\log k + k \log(2eM_h) \right\},
  \end{equation*}
  whenever $j\in N_x$, for a constant $K_1$ depending only on $f_0$, $\beta$ and
  $g$. Therefore, choosing $k \geq (2eM_h)^{2/(1-\gamma)}$, we deduce from
  \cref{eq:17,eq:13} that
  \begin{equation}
    \label{eq:12}
    \left| \int K_{hI}(x-y)(Q - P_k)(dy) \right|
    \leq 3|Q|h^{\beta+d} + K_1 \exp\left\{- \frac{1-\gamma}{2}k \log
      k\right\}.
  \end{equation}
  Now if $(2eM_h)^{2/(1-\gamma)}\geq 2(\beta+d)/(1-\gamma)\log h^{-1}$ set $k$
  to be the smaller integer larger than $(2eM_h)^{2/(1-\gamma)}$ ; otherwise set
  $k$ to be the larger integer greater than $2(\beta+d)/(1-\gamma)\log
  h^{-1}$. This yields the first part of the proposition with
  $f(x) = h^{-d}\int K_{hI}(x - y) P_k(dy)$ because of \cref{eq:12}, of
  \cref{pro:6} and because each of the $P_{j,k}$ has a number of atoms
  proportional to $(\log h^{-1})^{d \zeta}$ by Tchakaloff's theorem, all in
  $[-2S,2S]^d$ if $h$ is small enough.

  \par It remains to prove the separation between the atoms of $Q'_k$. But the
  cost to the supremum norm of moving one $\mu_i$ of $h^{\beta + 1}$ is
  proportional to $h^{\beta}$ by \cref{pro:11}. Hence we can assume that the
  support point of $Q_k'$ are chosen on a regular grid with $h^{\beta + 1}$
  separation within nodes (see also \citet[corollary~B1]{Shen2013}).
\end{proof}

\subsection{Kullback-Leibler property}
\label{sec:kullb-leibl-prop}

A simple computation shows that (see for instance \citet{Choi2007})
for $\theta_0=(f_0,\sigma_0)$ and $\theta = (f,\sigma)$,
\begin{gather*}
  K_{i}(\theta_0,\theta) = \log\frac{\sigma}{\sigma_0} - \frac 12 \left(1 -
    \frac{\sigma_0^2}{\sigma^2}\right) + \frac 12 \frac {|f_0(x_i) -
    f(x_i)|^2}{\sigma^2},\\
  V_{2;i}(\theta_0,\theta) = \frac 12 \left(1 -
    \frac{\sigma_0^2}{\sigma^2}\right)^2 +
  \frac{\sigma_0^4}{\sigma^4}|f_0(x_i)-f(x_i)|^2.
\end{gather*}
Therefore, for all $0 < \epsilon \leq 1/2$, there exists a constant $C_0>0$
(depending only on $\theta_0$) such that one has the inclusions
\begin{align}
  \label{eq:2b}
  K_n(\theta_0,\epsilon)
  &\supseteq \Set*{ (f,\sigma) \given \|f - f_0\|_{\infty}^2 \leq
    C_0\epsilon^2 ,\ \sigma_0 \leq \sigma \leq \sigma_0(1 + C_0\epsilon^2)},
\end{align}
hence probabilities of Kullback-Leibler balls around $\theta_0$ are lower
bounded by the probability of the sets defined in the rhs of \cref{eq:2b}. Now
we state and prove the main result of this section.

\begin{lemma}
  \label{lem:2}
  Let $f_0\in \mathcal{C}^{\beta}[-S,S]^d$, and $\zeta > 1$ as in
  \cref{pro:13}. Then there exists a constant $C>0$, not depending on $n$, such
  that $\Pi(K_n(\theta_0,\epsilon_n)) \gtrsim \exp(-n\epsilon_n^2)$ for
  \begin{equation*}
    \epsilon_n^2 = Cn^{-2\beta/(2\beta + d + \kappa/2)}(\log n)^{2\beta d(\zeta
      - 1)/(2\beta + d + \kappa/2)}.
  \end{equation*}
\end{lemma}
\begin{proof}
  By \cref{pro:13} for any $h>0$ sufficiently small, there is
  $N \lesssim h^{-d}(\log h^{-1})^{d(\zeta - 1)} $ and a function
  $f_h(x) = \sum_{j=1}^N\alpha_{j}\, K_{hI}(x - \mu_j)$ such that
  $|f_h(x) - f_0(x)| \lesssim h^{\beta}$ for all $x \in [-S,S]^d$, with
  $\alpha_{j} \in \reals$ for all $j=1,\dots N$, $\mu_i \in [-2S,2S]^d$ for all
  $i=1,\dots,N$, and $|\mu_i - \mu_j|_d \geq h^{\beta + 1}$ whenever $i\ne
  j$. Let define
  \begin{align*}
    \mathcal{E}_{s,h}
    &:= \Set*{A \in \mathcal{E}_s \given h^{-1} \leq \lambda_i(A^{-1})
      \leq h^{-1}(1 + h^{\beta+d}),\quad i=1,\dots,d}.
  \end{align*}
  We construct a partition of $\mathcal{E}_s\times [-2S,2S]^d$ in the following
  way : for all $j=1,\dots,N$, let $U_j$ be the closed ball of radius
  $h^{\beta + d +1}$ centered at $\mu_j$ (observe that these balls are
  disjoint), and set $V_j := \mathcal{E}_{s,h} \times U_j$,
  $V^c := \mathcal{E}_s\times [-2S,2S]^{d} \backslash U_{j=1}^NV_j$. Let
  $\mathcal{Q}$ denote the set of signed measures on
  $\mathcal{E}_{s}\times [-2S,2S]^{d}$ satisfying
  $Q\in \mathcal{Q} \Rightarrow |Q(V_j) - \alpha_{h,j}| \leq h^{\beta}N^{-1}$
  for all $j=1,\dots,N$, and $|Q|(V^c) \leq h^{\beta}$. Notice that for any
  $Q \in \mathcal{Q}$ we have
  $|Q| \leq \sum_{j=1}^N|Q(V_j) - \alpha_{h,j}| + \sum_{j=1}^N|\alpha_{h,j}|
  \lesssim h^{\beta} + h^{-d} \lesssim h^{-d}$ because of \cref{pro:13}. Then
  for any $Q\in \mathcal{Q}$ and all $x\in [-S,S]^d$, using \cref{pro:11},
  \begin{multline*}
    \left|\int_{\mathcal{E}_s\times[-2S,2S]^d}K_A(x - \mu)\, Q(dAd\mu) -
      f_h(x)\right| \lesssim  \sum_{j=1}^N|Q(V_j) - \alpha_{h,j}| + |Q|(V^c)\\
    + \sum_{j=1}^N\int_{V_j}\left|K_A(x - \mu) - K_{hI}(x - \mu_j)\right|\,
    |Q|(dAd\mu) \lesssim h^{\beta}.
  \end{multline*}
  Thus for all $Q \in \mathcal{Q}$ and all $x\in[-S,S]^d$, we have
  $|\int K_A(x - \mu)\,Q(dAd\mu) - f_0(x)| \leq |\int K_A(x - \mu)\,Q(dAd\mu) -
  f_h(x)| + |f_h(x) - f_0(x)| \leq K_{1}h^{\beta}$ for a constant $K_1> 0$ not
  depending on $h$. By the assumptions of \cref{eq:16,eq:14} we have for any
  $j=1,\dots,N$
  \begin{align*}
    \alpha F_A(\mathcal{E}_{s,h})F_{\mu}(U_j)
    &\geq \alpha b_1b_4 h^{a_1(\beta + d + 1)- a_4+ a_5(\beta +d)} \exp(-C_3
      h^{-\kappa/2})\\
    &=: K_2 h^{q} \exp(-C_3h^{-\kappa/2}),
  \end{align*}
  where $q := a_1(\beta + d +1)-a_4+ a_5(\beta +d)$ and the constant $K_2 > 0$
  not depending on $h$.

  \par For $h > 0$ sufficiently small, it is clear that
  $K_2 h^{q} \exp(-C_3h^{-\kappa/2}) < F(V_j) \leq 1$ for all $j=1,\dots,N$. We
  also assume without loss of generality that
  $K_2 h^{q} \exp(-C_3h^{-\kappa/2}) \leq F(V^c) \leq 1$ and we set
  $V_{N+1}:=V^c$, $\alpha_{h,N+1} :=0$ ; otherwise we subdivide $V^c$ onto
  smaller subsets for which the relation is verified. Because $F$ is a
  probability measure, this can be done with a finite number of subsets not
  depending on $h$. Now let
  $W := \Set{\sigma > 0 \given \sigma_0 \leq \sigma \leq \sigma_0(1 +
    C_0\epsilon_n^2)}$ and $\epsilon_n = C_0^{-1}K_1h^{\beta}$. Notice that
  $P^{\sigma}(W) \geq K_3\epsilon_n^{2a_9}$ with a constant $K_3>0$ eventually
  depending on $\theta_0$. The sets $\mathcal{E}_{s,h}\times U_j$ are disjoint,
  hence by \cref{eq:8b} and \cref{pro:5} we deduce that there is a constant
  $K_4 > 0$ such that
  \begin{align*}
    \Pi(K_n(\theta_0,\epsilon_n))
    &\geq P^{\sigma}(W)\Pi_{*}(\mathcal{Q})
      \gtrsim \epsilon_n^{2a_9} \prod_{i=1}^{N+1}\left(\frac{h^{\beta}N^{-1}
      \e^{-(3+\eta)|\alpha_j|}}{\Gamma(\alpha(F_A(\mathcal{E}_h)
      F_{\mu}(U_j))}\right)\\
    &\geq \exp\left\{- K_4 h^{-(d+\kappa/2)}(\log h^{-1})^{d(\zeta - 1)}
      \right\},
  \end{align*}
  where we used that $N \lesssim h^{-d}(\log h^{-1})^{d(\eta-1)}$,
  $\sum_{j=1}^N|\alpha_{h,j}| \lesssim h^{-d}$ and $\Gamma(x)\lesssim x^{-1}$
  for $x > 0$ sufficiently small. This concludes the proof.
\end{proof}

\section{Proof of \cref{thm:4}}
\label{sec:proofs-1}

As in \cref{sec:rates-convergence-loc-scale}, the proof of \cref{thm:4} consists
on verifying the condition established in \cref{thm:3}.

\subsection{Sieve construction}
\label{sec:sieve-construction}

For constants $H,M > 0$ to be determined later, we define
\begin{gather*}
  \small
  \Theta_{n} := \Set*{ (f,\sigma) \given
  \begin{array}{l}
    f(x) = \int K_{\xi,\phi}(x - \mu)Q(d\xi d\mu d\phi),\
    \supp Q = \mathbb R^d \times [-2S,2S]^d \times [0,\pi/2],\\
    Q = \sum_{i=1}^{\infty}u_i\delta_{\xi_i,\mu_i,\phi_i},\
    n^{-2/a_8} < \sigma^{2} \leq n^{-2/a_8}(1+M\epsilon_n)^n \\
    \sum_{i=1}^{\infty}|u_i|\leq n,\ \#\Set{i \given |u_i| > n^{-1},\ |\xi_i|_d
    \leq \e^{2n\epsilon_n^2}} \leq Hn\epsilon_n^2 / \log n,\\
    \sum_{i=1}^{\infty}|u_i|\Ind\Set{|\xi_i|_d  > \e^{2n\epsilon_n^2} } \leq
    M\epsilon_n,\quad \sum_{i=1}^{\infty}|u_i|\Ind\Set{|u_i| \leq n^{-1}} \leq
    M\epsilon_n
  \end{array}
}.
\end{gather*}

In the sequel, we assume without loss of generality that the jumps of $Q$ in the
definition of $\Theta_n$ are ordered so that there is no jump with
$|u_i| > n^{-1}$ and $|\xi_i|_d \leq \e^{2n\epsilon_n^2}$ when
$i>H n\epsilon_n^2 / \log n$. Moreover, we consider the following partition of
$\Theta_n$. Let $H_n$ be the largest integer smaller than
$Hn\epsilon_n^2 / \log n$. Then for any
$j = (j_1,\dots,j_{H_n})\in \Set{1,2,\dots}^{H_n}$ we define the slices
\begin{equation*}
  \Theta_{n,j} :=
  \Set*{(f,\sigma) \in \Theta_n \given
    \sqrt{n}(j_i -1) \leq |\xi|_d < \sqrt{n}j_i, \quad
    \forall i \leq H_n }.
\end{equation*}

\begin{lemma}
  \label{lem:1}
  Assume that there is $0 < \gamma_1 < 1$ such that
  $\epsilon_n^2 \geq n^{-\gamma_1}$ for all $n$ large enough. Then for
  $H = 6(1 - \gamma_1)^{-1}$ it holds
  $\Pi(\Theta_n^c) \lesssim \exp(-3n\epsilon_n^2)$ as $n\rightarrow \infty$.
\end{lemma}
\begin{proof}
  According to the proof of \cref{lem:9b}, the result holds if
  $F_{\xi}(\xi\ :\ |\xi_d| \geq \e^{2n\epsilon_n^2}) \lesssim
  \epsilon_n^2\exp(-3n\epsilon_n^2)$ for $n$ sufficiently large. Then the
  conclusion follows from \cref{eq:11} because $\eta > 0$.
\end{proof}

\begin{lemma}
  \label{lem:4}
  Let $\epsilon_n \rightarrow 0$ with $n\epsilon_n^2 \rightarrow \infty$ and
  $K = 3(32 \vee 4\sigma_0^2)$. Then there exists $M > 0$ such that it holds
  $\sum_{j}\sqrt{N(M\epsilon_n,\Theta_{n,j},\rho_n)} \sqrt{\Pi(\Theta_{n,j})}
  \e^{-(K M^2 - 2)n\epsilon_n^2} \rightarrow 0$.
\end{lemma}
\begin{proof}
  With the same argument as in \cref{lem:6}, it follows from \cref{eq:11} that
  for any $j \in \Set{1,2,\dots}^{H_n}$
  \begin{align*}
    \Pi(\Theta_{n,j})
    &\leq \textstyle \prod_{i\leq H_n} F_{\xi}(\xi\ :\ |\xi|_d \geq \sqrt{n}(j_i
      - 1) )
      \leq b_{11}^{H_n}\prod_{i\leq H_n} (1 + \sqrt{n}(j_i - 1))^{-2(\eta + 1)}.
  \end{align*}
  Moreover, using \cref{pro:8} we can find a constant $C> 0$ independent of $M$
  such that
  $N(M\epsilon_n,\Theta_{n,j},\rho_n) \leq \exp(2CHn\epsilon_n^2) \prod_{i\leq
    H_n}j_i^{d-1}$ when $n$ is large. Therefore, for those $n$
  \begin{multline*}
    \textstyle
    \sqrt{N(M\epsilon_n,\Theta_{n,j},\rho_n)}\sqrt{\Pi(\Theta_{n,j})}\\
    \leq
    \exp\left\{Hn\epsilon_n^2\left(C + \frac{\log b_{11}}{2\log n}\right)
    \right\} \prod_{i\leq H_n} j_i^{(d-1)/2}[1 + \sqrt{n}(j_i-1)]^{-(\eta+1)}.
  \end{multline*}
  For $n$ large enough we have $\log b_{11} \leq 2C\log n$ ; then provided
  $\eta > (d-1)/2$, we can sum over $j\in \Set{1,2,\dots}^{H_n}$ the last
  expression to get
  \begin{multline*}
    \textstyle \sum_{j} \sqrt{N(M\epsilon_n,\Theta_{n,j},\rho_n)}
    \sqrt{\Pi(\Theta_{n,j})}\\
    \begin{aligned}
      &\leq \exp\left\{2C Hn\epsilon_n^2\right\} \left( \textstyle \sum_{k\geq
          1} k^{(d-1)/2}[1 + \sqrt{n}(k -
        1)]^{-(\eta+1)} \right)^{H_n}\\
      &\leq \exp\left\{2CHn\epsilon_n^2 \right\} \left(1 +
        n^{-(\eta+1)/2}\textstyle \sum_{k\geq 1}k^{(d-1)/2 - (\eta+1)}
      \right)^{H_n} \lesssim \exp\{ 3CHn\epsilon_n^2 \},
    \end{aligned}
  \end{multline*}
  where the last inequality holds for $n$ sufficiently large. Now choose $M > 0$
  satisfying $KM^2 > 2 + 3CH$ to obtain the conclusion of the
  lemma.
\end{proof}

\begin{proposition}
  \label{pro:8}
  For $n$ large enough there is a constant $C > 0$ independent of $M$ such that
  for any sequence $\epsilon_n \rightarrow 0$ with
  $n\epsilon_n^2 \rightarrow \infty$, the following holds for any
  $j \in \Set{1,2,\dots}^{H_n}$.
  \begin{equation*}
    \log N(M\epsilon_n,\Theta_{n,j},\rho_n)
    \leq CH n \epsilon_n^2 + (d-1)\sum_{i\leq H_n}\log j_i.
  \end{equation*}
\end{proposition}
\begin{proof}
  The proof is similar to \cref{pro:4b}. Let $\widehat{R}_n$ be a
  $(M\epsilon_n/n)$-net of $[-2S,2S]^d$, $\widehat{\Delta}_n$ be a
  $M\epsilon_n$-net of
  $\Set{(u_1,\dots,u_{H_n}) \in \reals^{H_n}\given \sum_{i=1}^{H_n}|u_i| \leq
    n}$ in the $\ell_1$-distance,
  $\widehat{S}_n := \Set{\sigma > 0 \given \sigma^2 = n^{-2/a_8}(1 +
    M\epsilon_n)^k,\ k\in \mathbb N,\ k\leq n}$, $\widehat{U}_n$ be a
  $(M\epsilon_n/n)$-net of $[0,\pi/2]$, and for all $k=1,\dots,H_n$, let
  $\widehat{V}_{n,k}$ a $(M\epsilon_n/n)$-net of
  $\Set{\xi \in \mathbb R^d \given \sqrt{n}(k-1) \leq |\xi|_{d} < \sqrt{n} k }$.
  Pick $(f,\sigma) \in \Theta_{n,j}$ with
  $f(x) = \sum_{i=1}^{\infty}u_i\, K_{\xi_i,\phi_i}(x - \mu_i)$. Clearly we can
  find $\widehat{u} \in \widehat{\Delta}$ such that
  $\sum_{i\leq H_n}|u_i - \widehat{u}_i| \leq M\epsilon_n$,
  $\widehat{\mu} \in \widehat{R}_n^{H_n}$ such that
  $|\mu_i - \widehat{\mu}_i|_d \leq M\epsilon_n/n$ for all $i=1,\dots, H_n$,
  $\widehat{\phi} \in \widehat{U}_n^{H_n}$ such that
  $|\phi_i - \widehat{\phi}_i| \leq M \epsilon_n / n$ for all $i=1,\dots, H_n$,
  $\widehat{\xi}_i \in \widehat{V}_{n,j_i}$ such that
  $|\xi_i - \widehat{\xi}_i|_d \leq M\epsilon_n/n$ for all $i=1,\dots, H_n$, and
  $\widehat{\sigma} \in \widehat{S}_n$ such that
  $|\log \sigma - \log \widehat{\sigma}| \leq M\epsilon_n$. Let
  $\widehat{f}(x) = \sum_{i\leq H_n}\widehat{u}_i\,
  K_{\widehat{\xi}_i,\widehat{\phi}_{i}}(x - \widehat{\mu}_{i})$ denote the
  function built from the parameters chosen as above ; it follows
  \begin{align*}
    \|f - \widehat{f}\|_{2,n}
    &\leq
      \sum_{i> H_n}|u_i|
      + \sum_{i\leq H_n}|u_i - \widehat{u}_i|
      +\sum_{i\leq H_n}|u_i| \|K_{\xi_i,\phi_i}(\cdot -\mu_i) -
      K_{\widehat{\xi}_i,\widehat{\phi}_i}(\cdot - \widehat{\mu}_i)\|_{2,n}\\
    &\hspace{-1em}\leq
      2M\epsilon_n
      + C'\sum_{i\leq H_n}|u_i||\xi_i - \widehat{\xi}_i|_d
      + C'\sum_{i\leq H_n}|u_i||\mu_i - \widehat{\mu}_i|_d
      + C'\sum_{i\leq H_n}|u_i||\phi_i - \widehat{\phi}_i|\\
    &\hspace{-1em}\leq 2M(1+3C)\epsilon_n,
  \end{align*}
  for a constant $C'>0$ depending only on $g$, because of \cref{pro:11}. Thus a
  $2(1+3C')M\epsilon_n$-net of $\Theta_{n,j}$ in the distance $\rho_n$ can be
  constructed with $(\widehat{f},\widehat{\sigma})$ as above. Recall that
  $\# \widehat{R}_n \leq (4S n/(M\epsilon_n)^d$,
  $\# \widehat{\Delta}_n \leq (n/(M\epsilon_n))^{H_n}$, $\# \widehat{S}_n = n$,
  $\# \widehat{U}_n \leq \pi n/(2M\epsilon_n)$ and
  $\# \widehat{V}_k \leq (n^{3/2}k/(M\epsilon_n) + 1)^d -
  (n^{3/2}(k-1)/(M\epsilon_n) - 1)^d \lesssim (n^{3/2}/(M\epsilon_n))^d
  k^{d-1}$, where we used $u^d - v^d \leq d(u-v)u^{-1}$ for $v > u$. Then the
  end of the proof is identical to \cref{pro:4b}.
\end{proof}

\subsection{Approximation of functions}
\label{sec:appr-funct}

Let $\xi > 0$ and $m,r\geq 1$ be two positive integers. Let define the
approximating kernel $L_{m,\xi} : \reals^d\rightarrow \reals$ by the expression
\begin{equation*}
  L_{m,r}^{\xi}(x):= \lambda_{m,r}^{\xi}g(x)\prod_{i=1}^d\sin^{2r}(m \xi
  x_i)/\sin^{2r}(\xi x_i),
\end{equation*}
where $\lambda_{m,r}^{\xi} > 0$ is chosen so that
$\int_{\mathbb R^d}L_{m,r}^{\xi}(x)\,dx = 1$. Also let $\widetilde{f}_0$ denote
a suitable Whitney extension of $f_0$ from $[-S,S]^d$ to $\mathbb R^d$ (see the
proof of \cref{pro:6}). We may assume that $\widetilde{f}_0$ and all its
derivatives (up to order $\beta$) are zero outside $[-2S,2S]^d$. If it is not
the case, it suffices to multiply $\widetilde{f}_0$ by a smooth function that
equal $1$ on $[-S,S]^d$ and $0$ outside $[-2S,2S]^d$ (for instance, think about
the convolution of a bump function with a proper indicator set function).

In order to achieve good order of approximation of $f_0$ when $\beta$ is large,
we construct a transformation of $\widetilde{f}_0$ as follows.  In the sequel we
let $p$ be the largest integer strictly smaller than $\beta$. For all
multi-index $\alpha \in \mathbb N^d$, we define
$\mathfrak{m}_{\alpha}^{m,r,\xi} := \int_{\mathbb
  R^d}x^{\alpha}L_{m,r}^{\xi}(x)\,dx$. By definition of $L_{m,r}^{\xi}$, the
$\mathfrak{m}_{\alpha}^{m,r,\xi}$'s are always finite. Then we define
\begin{equation*}
  f_{\beta} \equiv f_{\beta}^{m,r,\xi} := \widetilde{f}_0 - \sum_{1 \leq
    |\alpha| \leq p} d_{\alpha} \mathfrak{m}_{\alpha}^{m,r,\xi}D^{\alpha}
  \widetilde{f}_0,
\end{equation*}
where the coefficients $(d_{\alpha})$ are defined in the same fashion as
\cref{eq:25}, with obvious modifications.

\begin{proposition}
  \label{pro:9}
  Let $m,r \geq 1$ be integers. For any $\beta > 0$ and any function
  $f_0 \in \mathcal{C}^{\beta}[-S,S]^d$ there is a constant $M_{\beta} > 0$ such
  that
  $|L_{m,\xi}^{r}*f_{\beta}(x) - f_0(x)| \leq M_{\beta}(\log m / m)^{\beta}$ for
  all $x\in [-S,S]^d$ if $2r \geq p + 1$ and $\xi = K_0(\log m)^{-1}$ for a
  constant $K_0$ depending only on $g$, $\beta$ and $r$.
\end{proposition}
\begin{proof}
  First assume $0 < \beta\leq 1$. By assumptions on $f_0$, there is $M > 0$ such
  that for all $x,y \in \mathbb R^d$ we have
  $|\widetilde{f}_0(x)-\widetilde{f}_0(y)| \leq M|x-y|_d^{\beta}$. Then,
  \begin{align*}
    \left|\widetilde{f}_0(x) - L_{m,r}^{\xi}* \widetilde{f}_0 (x)\right|
    &\leq\int_{\reals^d}\left|\widetilde{f}_0(x) -
      \widetilde{f}_0(y)\right| |L_{m,r}^{\xi}(x-y)|\,dy\\
    &\leq M\int_{\reals^d}|x - y|_d^{\beta}|L_{m,r}^{\xi}(x-y)|\,dy.
  \end{align*}
  Remark that for any $\tau > 0$ and all $u\in \mathbb R^d$ we have
  $\sum_{i=1}^d|u_i|^{\tau} \leq d \max_{i=1,\dots,d}|u_i|^{\tau} \leq
  d(\sum_{i=1}^d|u_i|^2)^{\tau/2}$ and
  $|x-y|_d^{\beta} = (\sum_{i=1}^d|x_i-y_i|^2)^{\beta/2} \leq
  d^{\beta/2}\max_{i=1,\dots,d}|x_i - y_i|^{\beta} \leq
  d^{\beta/2}\sum_{i=1}^d|x_i - y_i|^{\beta}$. Then, because
  $|g(x)| \lesssim \exp(-C_0|x|^{\tau}_d)$,
  \begin{align}
    \notag
    &\left|\widetilde{f}_0(x) - L_{m,r}^{\xi}* \widetilde{f}_0 (x)\right|\\
    \notag
    &\
      \lesssim \lambda_{m,r}^{\xi}\sum_{i=1}^d\int_{\reals^d}|u_i|^{\beta}
      \exp\left(-C_0d^{-1} \textstyle\sum_{j=1}^d|u_j|^{\tau}\right)
      \prod_{j=1}^d\frac{\sin^{2r}(m\xi u_j)}{\sin^{2r}(\xi u_j)}\, du\\
    \label{eq:3}
    &\ \lesssim
      \lambda_{m,r}^{\xi} \left(\int_{\mathbb R}
      |u|^{\beta}\e^{-C_0|u|^{\tau}/d} \frac{\sin^{2r}(m\xi u)}{\sin^{2r}(\xi
      u)}\,du\right) \left( \int_{\mathbb R}
      \e^{-C_0|u|^{\tau}/d}\frac{\sin^{2r}(m\xi u)}{\sin^{2r}(\xi u)}\,du
      \right)^{d-1}.
  \end{align}
  We now bound the first integral of the rhs of \cref{eq:3}. Let split the
  domain into three parts : $D_1 := (-1/(\xi m),1/(\xi m))$,
  $D_2 := [-1/(\xi m), -\pi/\xi] \cup [1/(\xi m), \pi/\xi]$ and
  $D_3 := \mathbb R\backslash (D_1\cup D_2)$. On $D_1$ and $D_3$ we always have
  $\sin^2(m\xi u) / \sin^2(\xi u) \lesssim m^2$, whereas on $D_2$ it holds
  $\sin^2(m \xi u) / \sin^2(\xi u) \lesssim 1/(\xi x)^2$. Therefore,
  \begin{align*}
    &\int_{\mathbb R} |u|^{\beta}\e^{-|u|^{\tau}/d}
      \frac{\sin^2(m\xi u)}{\sin^2(\xi u)}\,du\\
    &\ \lesssim
      m^{2r}\int_{D_1}|u|^{\beta}\,du
      + \xi^{-2r}\int_{D_2}|u|^{\beta-2r}\,du
      + m^{2r}\int_{D_3}|u|^{\beta}\e^{-C_0|u|^{\tau}/d}\,du =: I_1 + I_2 + I_3.
  \end{align*}
  The bounds $I_1 \lesssim m^{-\beta + (2r - 1)}\xi^{-(\beta+1)}$ and
  $I_2 \lesssim \xi^{-(\beta+1)}(1 + m ^{-\beta+(2r-1)})$ are obvious. Now we
  bound $I_3$. By Markov's inequality, for any $t < C_{0}/d$, we have
  \begin{equation*}
    \int_{\pi/\xi}^{\infty}u^{\beta}\exp(-C_0u^{\tau}/d)
    \leq \e^{-\pi/\xi}\int_0^{\infty}u^{\beta}\exp(-C_0u^{\tau}/d+ ut)\, du.
  \end{equation*}
  Now it is clear that $I_3 \lesssim m^{2r}\exp(-\pi/\xi)$ since by assumption
  $\tau > 1$ and we can choose $t < C_0/d$. It follows
  $I_3 \lesssim \xi^{-(\beta+1)}$ if $\xi = K_0(\log m)^{-1}$ for a suitable
  constant $K_0 > 0$ depending only on $g$, $\beta$ and $r$. The same reasoning
  applies to the second integral of the rhs of \cref{eq:3}, yielding the bound
  \begin{equation}
    \label{eq:10}
    \left|\widetilde{f}_0(x) - L_{m,\xi}* \widetilde{f}_0 (x)\right|
    \lesssim \lambda_{m,r}^{\xi} m^{-\beta + d(2r-1)}(\log m)^{\beta + d},
  \end{equation}
  whenever $\xi = K_0(\log m)^{-1}$. Hence, it remains to bound
  $\lambda_{m,r}^{\xi}$. By assumption, we have $g(x) \geq 0$ for all
  $x\in \mathbb R^d$ and a constant $C>0$ such that $g(x) > C$ on a set
  $E \subseteq [-\pi,\pi]^d$~; thus
  \begin{multline*}
    \frac{1}{\lambda_{m,r}^{\xi}} \geq \int_{E}g(x)
    \prod_{i=1}^d\frac{\sin^{2r}(m\xi x_i)}{\sin^{2r}(\xi x_i)}\,dx \gtrsim C
    m^{2dr}\int_E \prod_{i=1}^d \frac{\sin^{2r}(m\xi x_i)}{(m\xi x_i)^{2r}}\,
    dx\\
    \gtrsim
    \frac{m^{d(2r-1)}}{\xi^d}\int_{E'}\prod_{i=1}^d\frac{\sin^2(u_i)}{u_i^2}\,du
    \gtrsim \frac{m^{d(2r-1)}}{\xi^d},
  \end{multline*}
  where $E' := \Set{m\xi x \given x \in E}$ has non-null Lebesgue measure by
  assumption. Combining the last result with \cref{eq:10}, we get the estimate
  $|\widetilde{f}_0(x) - L_{m,\xi}*\widetilde{f}_0(x)| \lesssim m^{-\beta}(\log
  m)^{\beta}$ for all $x\in \mathbb R^d$ provided $\xi \leq K_0(\log m)^{-1}$.

  \par Now assume that $\beta > 1$.  Acting as in the previous paragraph, we can
  have $\mathfrak{m}_{\alpha}^{m,r,\xi} \leq m^{-|\alpha|}(\log m)^{|\alpha|}$
  for all $|\alpha|\leq p$, provided $2r > p + 1$ and $\xi = K_0'(\log m)^{-1}$
  for a suitable constant $K_0' > 0$. Then the proof is identical to
  \citet[lemma~2]{Shen2013}.
\end{proof}

\begin{proposition}
  \label{pro:2}
  Let $m\geq 1, r\geq (\beta +1)/2$ be integers and $\xi = K_0(\log m)^{-1}$,
  with $K_0$ as in \cref{pro:9}. There exists a discrete mixture
  $f(x) = \sum_{i=1}^N \alpha_i\, K_{\xi_i,\phi_{i}}(x-\mu_i)$ with
  $N \lesssim (m \log m)^d$ and for all $i=1,\dots,N$ : $\mu_i \in [-2S,2S]^d$,
  $\xi_i \in [0, 2rK_0m/\log m]^d$, $\phi_i \in [0,\pi/2]$~; such that
  $|f(x) - f_0(x)| \lesssim (\log m / m)^{\beta}$ for all $x\in
  [-S,S]^d$. Moreover $\sum_{i=1}^N|\alpha_i| \lesssim 1$, and for any $i\ne j$
  it holds $|\xi_i - \xi_j|_d \geq 2(\log m /m)^{\beta}$, $|\mu_i - \mu_j|_d
  \geq 2(\log m /m )^{\beta}$ and
  $|\phi_i - \phi_j| \geq 2(\log m /m )^{\beta}$.
\end{proposition}
\begin{proof}
  We rewrite $L_{m,r}^{\xi}$ in a more convenient form for the sequel. Let
  $a_0 := 1$ and $a_{k} = 2(1-k/m)$ for all $k=1,\dots,m - 1$. Then first step
  is to notice that
  \begin{equation*}
    L_{m,r}^{\xi}(x)
    = m^{dr} \lambda_{m,r}^{\xi}\, g(x)
      \prod_{i=1}^d\left[\sum_{k=0}^{m-1}a_k\cos(2\xi k x_i) \right]^r.
  \end{equation*}
  From here, letting $\mathcal{I}_r := \Set{0,\dots,m-1}^{r}$ and
  $\mathcal{S} = \Set{-1,1}$,
  \begin{equation*}
    L_{m,r}^{\xi}(x) = m^{dr}\lambda_{m,r}^{\xi}\, g(x) \prod_{i=1}^d
    \left[\sum_{k\in \mathcal{I}_r}a_k' 2^{-r}\sum_{e\in
        \mathcal{S}^r}\cos\left(2\xi x_i\textstyle \sum_{j=1}^re_j k_j\right)
    \right],
  \end{equation*}
  where $a_k' := a_{k_1}\dots a_{k_r}$, and because
  $\prod_{j=1}^r\cos(2\xi k_j x_i) = 2^{-r}\sum_{e\in
    \mathcal{S}^r}\cos(2\xi\sum_{j=1}^r e_jk_jx_i)$. Notice that
  $|a_k'|2^{-r} \leq 1$ for all $k \in \mathcal{I}_r$, and that
  $2|\sum_{j=1}^re_jk_j|$ can take at most $1+ r(m-1)$ values~; we denote these
  unique values $\omega_j$ with $j \in \mathcal{J} :=
  \Set{0,\dots,r(m-1)}$. Then we can rewrite,
  \begin{align*}
    L_{m,r}^{\xi}(x) = m^{dr}\lambda_{m,r}^{\xi}\, g(x) \prod_{i=1}^d
    \left[\sum_{k\in \mathcal{J}}a_k'' \cos(\xi \omega_j x_i ) \right],
  \end{align*}
  where the coefficients $a_k''$ satisfy
  $|a_k''| \leq 2 \#( \mathcal{I}_r\times \mathcal{S}^r) \leq 2(2m)^r$. Finally,
  for all $k \in \mathcal{J}^d$ letting
  $b_k := 2^{-d}a''_{k_1}\dots a''_{k_{d}}$ and $\omega_{k,i} := \omega_{k_i}$,
  with the same arguments as previously,
  \begin{equation*}
    L_{m,r}^{\xi}(x)
    = m^{dr}\lambda_{m,r}^{\xi}\, g(x) \sum_{k\in \mathcal{J}^d}\sum_{e\in
      \mathcal{S}^d} b_k \cos\left(
      \xi \textstyle\sum_{i=1}^d \omega_{k,i}e_ix_i \right),
  \end{equation*}
  where $|b_k| \leq (2m)^{dr}$ for all $k \in \mathcal{J}^d$. Therefore,
  \begin{align*}
    &(m^{dr}\lambda_{m,r}^{\xi})^{-1}L_{m,r}^{\xi}*f_{\beta}(x)\\
    &\qquad = \sum_{k\in \mathcal{J}^d}\sum_{e\in
      \mathcal{S}^d} b_k \int_{\mathbb R^d} f_{\beta}(y) g(x-y)
      \cos\left(\xi
      \textstyle\sum_{i=1}^d\omega_{k,i}e_{i}(x_i-y_i)\right)\,dy\\
    &\qquad = \sum_{k\in \mathcal{J}^d}\sum_{e\in
      \mathcal{S}^d}
      b_k \cos\left( \xi \textstyle \sum_{i=1}^d\omega_{k,i}e_ix_i \right)
      \int_{\mathbb R^d}f_{\beta}(y)g(x-y)
      \cos\left( \xi \textstyle \sum_{i=1}^d\omega_{k,i}e_iy_i \right)
      \,dy\\
    &\qquad\quad+
      \sum_{k\in \mathcal{J}^d}\sum_{e\in \mathcal{S}^d}
      b_k \sin\left( \xi \textstyle \sum_{i=1}^d\omega_{k,i}e_ix_i \right)
      \int_{\mathbb R^d}f_{\beta}(y)g(x-y)
      \sin\left( \xi \textstyle \sum_{i=1}^d\omega_{k,i}e_iy_i \right)
      \,dy.
  \end{align*}
  We finish the proof by discretizing the integrals in the last
  equation. Obviously the proof are identical for both integrals, hence we only
  consider the first one. To ease notations, we set
  $h_{k}(x) := f_{\beta}(x)\prod_{i=1}^d\cos(\xi
  \sum_{i=1}^d\omega_{k,i}e_ix_i)$. For any integer $q \geq 1$, proceed as in
  the proof of \cref{pro:13} to find a signed measure
  $P_{k,q} =: \sum_{l\in \mathcal{L}} p_{k,l}\delta_{x_{k,l}}$ such that
  $\int_{[-2S,2S]^d}R(x)\, dP_{k,q}(x) = \int_{[-2S,2S]^d}R(x)\,h_{k}(x)dx$ for
  all polynomials $R(x)$ of degree $\leq q$, with
  $\# \mathcal{L} \leq (q + d)!/(q! d!)$ and
  $\sum_{l\in \mathcal{L}}|p_{k,l}| = \int_{[-2S,2S]^d}|h_{k}(x)|\,dx \leq M$
  for a positive constant $M$ (recall that by construction of $f_{\beta}$, we
  have $\|f_{\beta}\|_{\infty} < +\infty$, and
  $\supp f_{\beta} \subseteq [-2S,2S]^d$). Then for any $x\in \reals^d$,
  \begin{multline}
    \label{eq:8}
    \left| \int_{\mathbb R^d}h_{k}(y)g(x-y)\,dy - \int_{[-2S,2S]^d}
      g(x-y)\,dP_{k,q}(y)\right|\\
    \leq\sum_{|\alpha| \leq
      r}\frac{|D^{\alpha}g(0)|}{\alpha!}\left|\int_{[-2S,2S]^d}
      (x-y)^{\alpha}\,h_{k}(y)dy -
      \int_{[-2S,2S]^d}(x-y)^{\alpha}\,dP_{k,q}(y)
    \right|\\
    +\int_{[-2S,2S]^d}|R_q(y)|\,|h_{k}(y)|dy +
    \int_{[-2S,2S]^d}|R_q(y)|\,d|P_{k,q}|(y),
  \end{multline}
  where $|R_q(y)| \leq \sup_{|\alpha|=q}|D^{\alpha}g(0)| |y|_d^q/q!$. The first
  term of the rhs of \cref{eq:8} is null by construction of $P_{k,q}$. As in the
  proof of \cref{pro:13}, the two last terms of \cref{eq:8} are bounded by a
  constant multiple of
  \begin{gather*}
    \exp\left\{-(1-\gamma)q\log q + q(1+\log(2\sqrt{d}S))\right\}.
  \end{gather*}
  Then the error of approximating the integrals is $o(m^{-\beta})$ if
  $q = K_1\log m$ for a suitable constant $K_1 > 0$ depending only on $\beta$
  and $\gamma$. Since for $\xi = K_0'(\log m)^{-1}$ we have,
  \begin{equation*}
    m^{dr}\lambda_{m,r}^{\xi}\sum_{k\in \mathcal{J}^d}\sum_{e\in
      \mathcal{S}^d}|b_k|
    \lesssim m^{dr}\times m^{-d(2r-1)} \times (\log m)^{-d}\times
    \#\mathcal{J}^d\times m^{dr} \lesssim (\log m)^{-d},
  \end{equation*}
  the error of approximating $L_{m,r}^{\xi}* f_{\beta}$ by the discretized
  version does not exceed $o(m^{-\beta})$ when $q=K_1 \log m$. The conclusion of
  the proposition follows from elementary manipulation of trigonometric
  functions and because $\#\mathcal{L} \lesssim q^d \lesssim (\log m)^d$.

  It remains to prove the separation between the atoms of the mixing measure,
  but this follows from \cref{pro:12} with the same argument as in
  \cref{pro:13}.
\end{proof}

\subsection{Kullback-Leibler condition}
\label{sec:kullb-leibl-cond}

\begin{lemma}
  \label{lem:2b}
  Let $f_0\in \mathcal{C}^{\beta}[-S,S]^d$. Then there exists a constant $C>0$,
  not depending on $n$, such that
  $\Pi(K_n(\theta_0,\epsilon_n)) \gtrsim \exp(-n\epsilon_n^2)$ for
  $\epsilon_{n}^2 = Cn^{-2\beta/(2\beta +d)}(\log n)^{2\beta(2d+1)/(2\beta+d)}$.
\end{lemma}
\begin{proof}
  Let $f_m(x) = \sum_{i=1}^N\alpha_i\,K_{\xi_i,\phi_i}(x - \mu_i)$ be as in
  \cref{pro:2}. For any $i=1,\dots,N$ define the sets
  $U_i := \Set{\xi \in \mathbb R^d \given |\xi - \xi_i|_d \leq (\log m /
    m)^{\beta}}$, $V_i := \Set{\mu \in [-2S,2S]^d \given |\mu - \mu_i|_d \leq
    (\log m /m )^{\beta}}$ and $W_i := \Set{\phi \in [0,\pi/2] \given |\phi -
    \phi_i| \leq (\log m / m)^{\beta}}$. Notice that these sets are disjoint,
  and for any $i=1,\dots,N$ we have
  \begin{equation*}
    \alpha F(U_i\times V_i \times W_i)
    \gtrsim |\xi_i|_d^{-a_{12}}(\log m /m)^{\beta(a_1 + a_{10} + a_{13})}
    \gtrsim (\log m /m)^{q},
  \end{equation*}
  where $q:= da_{12} + \beta(a_1+a_{10}+a_{13})$. Then proceed as in
  \cref{lem:2}, to find constants $K_1,K_4 >0$ such that with
  $\epsilon_n = C_0^{-1}K_1 (\log m / m)^{\beta}$,
  \begin{equation*}
    \Pi(K_n(\theta_0,\epsilon_n))
    \geq \exp\left\{-K_4 m^d(\log m)^{d+1} \right\}.\qedhere
  \end{equation*}
\end{proof}


\appendix

\section{Symmetric Gamma distribution}
\label{sec:symm-gamma-distr}

The symmetric Gamma distribution $\mathrm{SGa}(a,b)$, with $a,b >0$ is the
distribution having Fourier transform $t \mapsto (1 + t^2/b^2)^{-a}$. It is
easily seen that if $X \sim \Ga(a,b)$ and $Y \sim \Ga(a,b)$, with $X$ and $Y$
independent, then $X - Y$ has $\mathrm{SGa}(a,b)$ distribution.

\begin{proposition}
  \label{pro:1}
  Let $Z\sim \mathrm{SGa}(a,b)$. Then for any positive integer $n$,
  \begin{align*}
    \Esp Z^{2n} = \frac{(2n)!}{n!}\frac{(a)^{(n)}}{b^{2n}},
    \qquad
    \Esp Z^{2n+1} = 0.
  \end{align*}
  Moreover, the distribution $\mathrm{SGa}(a,b)$ is determined by its moments
  (in the sense that $\mathrm{SGa}(a,b)$ is the only distribution with this
  sequence of moments).
\end{proposition}
\begin{proof}
  From definition of $\mathrm{SGa}(a,b)$, the random variable $Z$ is distributed
  as $X - Y$, where $X,Y \sim \mathrm{Ga}(a,b)$ and $X,Y$ are independent. Then
  it is obvious that all odd moments must vanish. For the even moments, we
  write,
  \begin{align*}
    \Esp(X-Y)^{2n}
    &= \sum_{k=0}^{2n}\binom{2n}{k}(-1)^k \Esp X^{2n-k}\Esp Y^k\\
    &= \frac{1}{b^{2n}}\sum_{k=0}^{2n}\binom{2n}{k}(-1)^k
      (a)^{(2n-k)}(a)^{(k)}
    = \frac{(2n)!}{n!}\frac{(a)^{(n)}}{b^{2n}},
  \end{align*}
  where the last equality can be obtained after some algebra. To see that
  $\mathrm{SGa}(a,b)$ is determined by its moments, we check that Carleman's
  criteria applies \citep{Gut2006}, which is straightforward.
\end{proof}

\begin{proposition}
  \label{pro:5}
  Let $X \sim \mathrm{SGa}(\alpha,\eta)$, with $0 < \alpha \leq 1$ and
  $\eta > 0$. Then there is a constant $C>0$ such that for any $x \in \mathbb R$
  and any $0< \delta \leq (3+\eta)^{-1}$ we have
  $\Pr(|X - x| \leq \delta) \geq C\delta \e^{-(3+\eta)|x|}\Gamma(\alpha)^{-1}$.
\end{proposition}
\begin{proof}
  Assume for instance that $x \geq 0$. Recalling that $X$ is distributed as the
  difference of two independent $\mathrm{Ga}(\alpha,\eta)$ distributed random
  variables, it follows
  \begin{align*}
    \Pr(|X - x| \leq \delta)
    &\geq
      \frac{1}{\Gamma(\alpha)} \int_0^{\infty} y^{\alpha -
      1}\e^{-\eta y}
      \frac{1}{\Gamma(\alpha)}
      \int_{x+y}^{x+y + \delta} z^{\alpha-1}\e^{-\eta z}\, dz\, dy.
  \end{align*}
  Because $\alpha \leq 1$, the mapping $z \mapsto z^{\alpha-1}e^{-\eta z}$ is
  monotonically decreasing on $\mathbb R^+$, then the last integral in the rhs
  of the previous equation is lower bounded by
  $\delta (x+y+\delta)^{\alpha -1}\e^{-\eta(x+y+\delta)} \geq \delta
  \e^{-(3+\eta)(x+y+\delta)}$. Then
  \begin{align*}
    \Pr(|X - x| \leq \delta)
    &\geq
      \frac{\delta \e^{-(3+\eta)(x+\delta)}}{\Gamma(\alpha)^2} \int_0^{\infty}
      y^{\alpha - 1}\e^{-(3+2\eta)y}\, dy\\
    & =\frac{\delta \e^{-(3+\eta)(x+\delta)}}{(3+2\eta)^{\alpha}\Gamma(\alpha)}
      \geq \frac{\delta \e^{-(3+\eta)|x|}}{\e(3+2\eta)^{\alpha}\Gamma(\alpha)}.
  \end{align*}
  The proof when $x < 0$ is obvious.
\end{proof}

\section{Auxiliary results}
\label{sec:auxiliary-results}

\begin{proposition}
  \label{pro:11}
  Let $K_A(x) = g(A^{-1}x)$, and assume that for all multi-index
  $k\in \integers^d$ with $|k| = 0,1,2$ the mapping $x\mapsto x^k g(x)$ belongs
  to $L^1(\reals^d)$. Let $\|\cdot\|$ be the spectral norm on
  $\mathcal{E}$. Then there is a constant $C>0$ such that for all
  $x,\mu_1,\mu_2 \in \reals^d$ and all $A_1,A_2 \in \mathcal{E}$ arbitrary with
  $\|I - A_1^{-1}A_2\| \wedge \|I - A_2^{-1}A_1\|$ small enough,
  \begin{multline*}
    |K_{A_1}(x - \mu_1) - K_{A_2}(x - \mu_2)|
    \leq
    C\|I - A_1^{-1}A_2\| \wedge C\|I - A_2^{-1}A_1\|\\
    + C\left(\|A_1^{-1}\| \wedge \|A_2^{-1}\| \right)|\mu_1 -
    \mu_2|_d.
  \end{multline*}
\end{proposition}
\begin{proof}
  Starting from the triangle inequality, we have
  \begin{multline}
    \label{eq:28}
    |K_{A_1}(x - \mu_1) - K_{A_2}(x - \mu_2)| \leq
    |K_{A_1}(x - \mu_2) - K_{A_2}(x - \mu_2)|\\
    + |K_{A_1}(x - \mu_1) - K_{A_1}(x - \mu_2)|
  \end{multline}
  We recall that $K_A(x) := g(A^{-1}x)$. To bound the first term, it is enough
  to bound $g(x) - g(A_1^{-1}A_2x)$ for all $x\in\reals^d$. Let
  $(B_n)_{n\in\integers}$ and $(C_n)_{n\in\integers}$ be two arbitrary sequences
  in $\mathcal{E}$ such that $\|I - B_n^{-1}C_n\| \leq 1/n$, and let
  $\widehat{g}$ denote the Fourier transform of $g$. Then,
  \begin{align*}
    \sup_{x\in\reals^d}\left|g(x) - g(B_n^{-1}C_nx)\right|
    & \leq \int_{\reals^d}\left| \widehat{g}(\xi) - |\det
      (B_n^{-1}C_n)|\, \widehat{g}(B_n^{-1}C_n\xi)\right|\,
      d\xi.
  \end{align*}
  Remark that $|\det B_n^{-1}C_n| \leq 1 + |\det(I - B_n^{-1}C_n)|$, and
  $\|I - B_n^{-1}C_n\| \leq 1/n$ implies that
  $|\det(I - B_n^{-1}C_n)| \leq \sqrt d/n^d$. Also,
  $|B_n^{-1}C_n\xi|_d \leq \|I - B_n^{-1}C_n\||\xi|_d + |\xi|_d \leq (1 +
  1/n)|\xi|_d$. It turns out that,
  \begin{align*}
    \lim_{n\rightarrow\infty}|\det B_n^{-1}C_n|\, \widehat{g}(B_n^{-1}C_n\xi) =
    \widehat{g}(\xi).
  \end{align*}
  We now prove that
  $\Set{|\det B_n^{-1}C_n|\, \widehat{g}(B_n^{-1}C_n\xi) \given n \geq 2}$ is
  dominated. By assumption, $g\in L^1(\reals^d)$, as well as
  $x\mapsto x^{k} g(x)$ with $|k|=1,2$. This implies that
  $|\widehat{g}(\xi)| \leq C(1 + |\xi|_d)^{-2}$ for some $C> 0$. We already saw
  that $|\det B_n^{-1}C_n| \leq 1 + 1/n^d$, and
  $|\xi|_d \leq |B_n^{-1}C_n\xi|_d + |(I - B_n^{-1}C_n)\xi|_d$ implies
  $|B_n^{-1}C_n\xi|_d \geq (1 - 1/n)|\xi|_d$. Therefore, for any $n \geq 2$,
  \begin{align*}
    |\det B_n^{-1}C_n|\, \widehat{g}(B_n^{-1}C_n\xi)
    &\leq \frac{C |\det B_n^{-1}C_n|}{(1 + |B_n^{-1}C_n\xi|_d)^2}
      \leq \frac{C(1 + 2^{-d})}{(1 + |\xi|_d/2)^2}.
  \end{align*}
  Then the dominated convergence applies, and
  \begin{align*}
    \lim_{n\rightarrow \infty}\sup_{x\in\reals^d}|g(x) - g(B_n^{-1}C_nx)| = 0.
  \end{align*}
  The second term of the rhs of \cref{eq:28} is bounded above by
  $|A_1^{-1}(\mu_1 - \mu_2)|_d \leq \|A_1^{-1}\|\,|\mu_1 - \mu_2|_d$, using
  Lipshitz continuity of $g$. Using a symmetry argument, the conclusion of the
  proposition follows.
\end{proof}

\begin{proposition}
  \label{pro:12}
  Let $K_{\xi,\phi}(x) = g(x)\cos(\sum_{i=1}^d \xi_ix_i + \phi)$, and assume
  that for all multi-index $k\in \integers^d$ with $|k| \leq 1$ we have
  $\sup_{x\in\reals^d}|x^kg(x)| < +\infty$ and
  $\sup_{x\in\reals^d}|D^kg(x)|$. Then there is a constant $C>0$ such that for
  all $x,\mu_1,\mu_2,\xi_1,\xi_2 \in \reals^d$ and all
  $\phi_1,\phi_2 \in [0,\pi/2]$
  \begin{align*}
    |K_{\xi_1}(x - \mu_1) - K_{\xi_2}(x - \mu_2)|
    \leq C|\xi_1 - \xi_2|_d + C|\mu_1 - \mu_2|_d + C|\phi_1 - \phi_2|.
  \end{align*}
\end{proposition}
\begin{proof}
  We write,
  \begin{multline*}
    |K_{\xi_1,\phi_1}(x - \mu_1) - K_{\xi_2,\phi_2}(x - \mu_2)|
    \leq |K_{\xi_1,\phi_1}(x - \mu_1) - K_{\xi_1,\phi_1}(x - \mu_2)|\\
    + |K_{\xi_1,\phi_1}(x - \mu_2) - K_{\xi_1,\phi_2}(x - \mu_2)|
    + |K_{\xi_1,\phi_2}(x - \mu_2) - K_{\xi_2,\phi_2}(x-\mu_2)|.
  \end{multline*}
  Because $g$ has bounded first derivatives, it is Lipschitz continuous for some
  Lipschitz contant $K > 0$, then the first term of the rhs is bounded above by
  $K |\mu_1 - \mu_2|_d$. With the same argument, the second term is bounded by a
  constant multiple of $\|g\|_{\infty}|\phi_1 - \phi_2|$. The last term of the
  rhs is easily bounded, because for all $x\in\reals^d$:
  \begin{align*}
    |K_{\xi_1,\phi_2}(x) - K_{\xi_2,\phi_2}(x)|
    &\leq |\cos(\textstyle\sum_{i=1}^d\xi_{1,i} x_i + \phi_2) - \cos(\textstyle
      \sum_{i=1}^d\xi_{2,i} x_i + \phi_2)||g(x)|\\
    &\leq \sum_{i=1}^d|\xi_{1,i}x_i - \xi_{2,i}x_i||g(x)|\\
    &\leq \left(\sum_{i=1}^d |\xi_{1i} - \xi_{2i}|^2\right)^{1/2}
      \left(\sum_{i=1}^d |x_i g(x)|^2\right)^{1/2},
  \end{align*}
  where the last line holds by Hölder's inequality. Then the conclusion follows
  $x\mapsto x^kg(x)$ is bounded for all $|k| = 1$.
\end{proof}

\begin{proposition}
  \label{pro:7}
  Let $g(x) = \exp(-|x|_d^2/2)$. Then
  $\sup_{x\in \mathbb R^d}|D^{\alpha}g(x)| \lesssim \exp(\frac 12|\alpha| \log
  |\alpha|)$ for all $\alpha \in \mathbb N^d$.
\end{proposition}
\begin{proof}
  For any $\alpha \in \mathbb N^d$, let $k = |\alpha| =
  \sum_{i=1}^d\alpha_i$. When $k < 2$, the result is obvious. Now assume that
  $k\geq 2$. By Fourier duality, we have for all $x\in \mathbb R^d$
  \begin{align*}
    |D^{\alpha}g(x)|
    &\leq \int |u^{\alpha}g(u)|\, du
      \lesssim 2^{k/2}\prod_{i=1}^d\Gamma\left(\frac{\alpha_i + 1}{2}\right)
      \lesssim
      2^{k/2}\prod_{i=1}^d\left(\frac{2}{\alpha_i+1}\right)\left(\frac{\alpha_i
      + 1}{2e}\right)^{\frac{\alpha_i+1}{2}},
  \end{align*}
  where the last inequality follows from Stirling formula. Then it is clear
  that,
  \begin{align*}
    |D^{\alpha}g(0)|
    &\lesssim
      \exp\left\{-\frac{k}{2} - \frac 12\sum_{i=1}^d\log(1 + \alpha_i)
      + \frac 12 \sum_{i=1}^d\alpha_i\log(1 + \alpha_i) \right\}.
  \end{align*}
  The result follows because for all $k \geq 2$ we have
  $\sum_{i=1}^d\alpha_i \log(1+\alpha_i) \leq \sum_{i=1}^d\alpha_i \log(1 + k)
  \leq (1/2 + \log k) \sum_{i=1}^d\alpha_i \leq k/2 + k \log k$.
\end{proof}

\section*{Acknowledgments}

The authors are grateful to Judith Rousseau and Trong Tuong Truong for their
helpful support and valuable advice throughout the writing of this article, and
also to Pr. Robert L. Wolpert for discussions and RJMCMC source code.

\bibliographystyle{abbrvnat}
\bibliography{bib/biblio-paper}

\end{document}